\documentclass[11pt]{article}

\usepackage[a4paper,margin=1in]{geometry}
\usepackage[T1]{fontenc}
\usepackage{lmodern}
\usepackage{amsmath,amsthm,amssymb,bm}
\usepackage{mathrsfs}
\usepackage{stmaryrd}
\usepackage{enumitem}
\usepackage{commath}
\usepackage{cases}
\usepackage{graphicx}
\usepackage{xcolor}
\usepackage{algorithm}
\usepackage{algpseudocode}
\usepackage{cite}
\usepackage{subcaption}
\usepackage{microtype}
\usepackage[colorlinks,allcolors=blue!60!black]{hyperref}
\newtheorem{theorem}{Theorem}[section]
\newtheorem{lemma}[theorem]{Lemma}
\newtheorem{proposition}[theorem]{Proposition}
\theoremstyle{remark}
\newtheorem{remark}[theorem]{Remark}
\numberwithin{algorithm}{section}
\numberwithin{equation}{section}

\newenvironment{keywords}{\par\medskip\noindent\textbf{Key words.}\ }{\par}
\newenvironment{MSCcodes}{\par\smallskip\noindent\textbf{MSC codes.}\ }{\par}
\newcommand{\email}[1]{\href{mailto:#1}{\nolinkurl{#1}}}

\allowdisplaybreaks[1]

\newcommand*\diff{\mathop{}\!\mathrm{d}}
\newcommand{\id}{\mathrm{id}}

\title{\bfseries An energy-stable parametric finite element framework for surface
diffusion with moving contact lines on curved substrates}

\author{Harald Garcke\thanks{Fakult\"at f\"ur Mathematik, Universit\"at
Regensburg, 93040 Regensburg, Germany (\email{harald.garcke@ur.de}).}
\and Robert N\"urnberg\thanks{Dipartimento di Matematica, Universit\`a di
Trento, 38123 Trento, Italy (\email{robert.nurnberg@unitn.it}).}
\and Rong Tang\thanks{Department of Applied Mathematics, The Hong Kong
Polytechnic University, Hung Hom, Hong Kong
(\email{claire.tang@polyu.edu.hk}).}}

\date{}
\renewcommand{\and}{\end{tabular}\hskip 2em\begin{tabular}[t]{c}}

\hypersetup{
  pdftitle={An energy-stable parametric finite element framework for surface diffusion with moving contact lines on curved substrates},
  pdfauthor={H. Garcke, R. N\"urnberg, and R. Tang}
}

\begin{document}

\maketitle

\begin{abstract}
We develop energy-stable parametric finite element methods for surface
diffusion of solid films with moving contact lines on curved substrates.
Discretizing both the film surface and the substrate patches allows the introduction of linear schemes that are unconditionally energy stable. A nonlinear variant of the schemes ensures that, in addition, the vertices of the substrate approximations remain attached to the true substrate throughout. Moreover, our framework can incorporate anisotropic surface energy densities, for both the film surface and the substrate, as well as different tangential motions.
Numerical experiments in two and three space dimensions illustrate the practicality of the introduced methods.
\end{abstract}



\section{Introduction}
\label{sec:intro}

In solid-state dewetting, a thin solid film on a substrate changes its shape
when it is annealed below its melting temperature.  Mass moves along surfaces
and interfaces rather than through the bulk, and the film may break up into
islands.  The process can damage microelectronic interconnects, and it can
also be used to make ordered nanoparticles for catalysis, sensing, and
nanowire growth; see the reviews of Thompson \cite{Thompson2012} and Leroy et
al.\ \cite{Leroy2016}.

In the sharp-interface model of Wang, Jiang, Bao, and Srolovitz
\cite{WangJiangBaoSrolovitz2015} for curves, extended to surfaces in
\cite{JiangZhaoBao2020} and to strongly anisotropic surface energies in
\cite{JiangWangZhaoSrolovitzBao2016}, the film surface \(\Gamma(t)\) evolves
by surface diffusion \cite{Mullins1957,CahnTaylor1994}, and its boundary
\(\partial\Gamma(t)\), the contact line, moves on the substrate
\(\Gamma_{\rm base}\).
The substrate is the disjoint union of the part
\(\Gamma_{\rm in}(t)\) wetted by the film and the dry part
\(\Gamma_{\rm out}(t)\).  With constant surface energy densities \(\gamma_0\),
\(\gamma_{\rm in}\), and \(\gamma_{\rm out}\) for the film/vapour,
film/substrate, and vapour/substrate interfaces, the total energy is
\begin{align}\label{eq:intro-energy}
  W=\gamma_0|\Gamma|+\gamma_{\rm in}|\Gamma_{\rm in}|
    +\gamma_{\rm out}|\Gamma_{\rm out}| .
\end{align}
The model reads
\[
  v\cdot n= \beta\gamma_0\,\Delta_\Gamma H\ \text{ on }\Gamma,\quad
  v\cdot n_b=0,\
  \gamma_0\,\mu_{\scriptscriptstyle\partial}\cdot\mu_b=\gamma_{\rm out}-\gamma_{\rm in},\
  \mu_{\scriptscriptstyle\partial}\cdot\nabla_\Gamma H=0\ \text{ on }\partial\Gamma .
\]
Here \(v\) is the velocity of \(\Gamma\), so that \(v\cdot n\) is its normal
velocity, \(H\) is the mean curvature, \(\beta>0\) is a surface mobility constant \cite{Mullins1957}, \(n\) and \(n_b\) are the unit normals
of \(\Gamma\) and \(\Gamma_{\rm base}\), and \(\mu_{\scriptscriptstyle\partial}\)
and \(\mu_b\) are the outward conormals of \(\Gamma\) and of
\(\Gamma_{\rm in}\) along \(\partial\Gamma\); see
Figure~\ref{fig:domain_2D}.
We assume \(|\gamma_{\rm out}-\gamma_{\rm in}|<\gamma_0\).
Then the second condition on \(\partial\Gamma\) is Young's law
\(\mu_{\scriptscriptstyle\partial}\cdot\mu_b=\cos\theta\), with the Young
angle \(\theta\in(0,\pi)\) given by
\(\cos\theta=(\gamma_{\rm out}-\gamma_{\rm in})/\gamma_0\).  The first keeps
the film attached to the substrate, and the third prevents mass flux.  The evolution can be interpreted as the \(H^{-1}\)-gradient flow of \(W\). In particular, it decreases the total energy, and so discretization methods should try to mimic this energy decay on the discrete level.
\begin{figure}[htbp]
  \centering
  \begin{subfigure}[b]{0.49\textwidth}
    \centering
    \includegraphics[width=\textwidth]{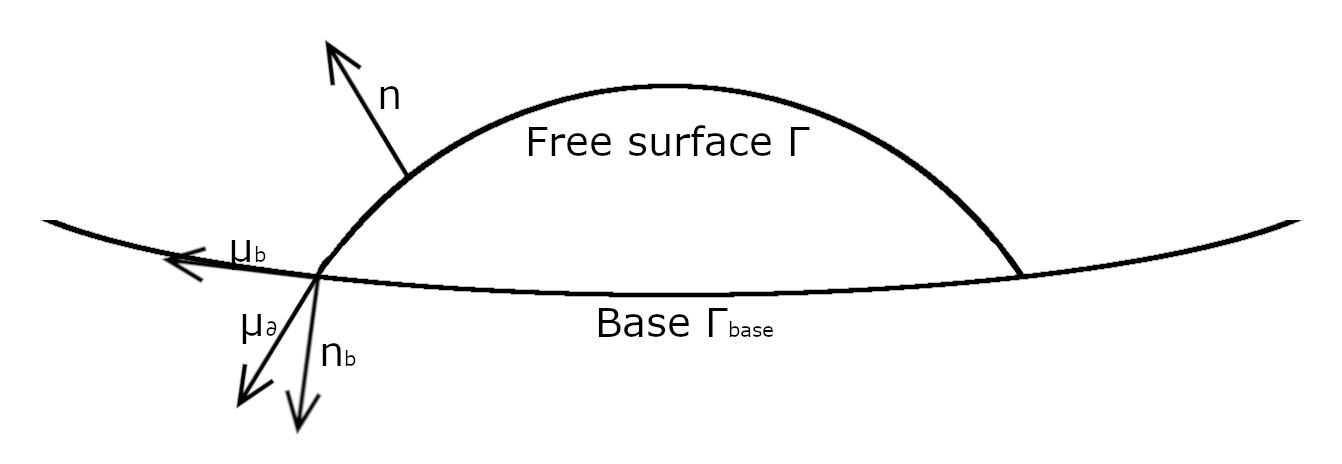}
    \caption{2D configuration}\label{domain}
  \end{subfigure}
  \hfill
  \begin{subfigure}[b]{0.49\textwidth}
    \centering
    \includegraphics[width=\textwidth]{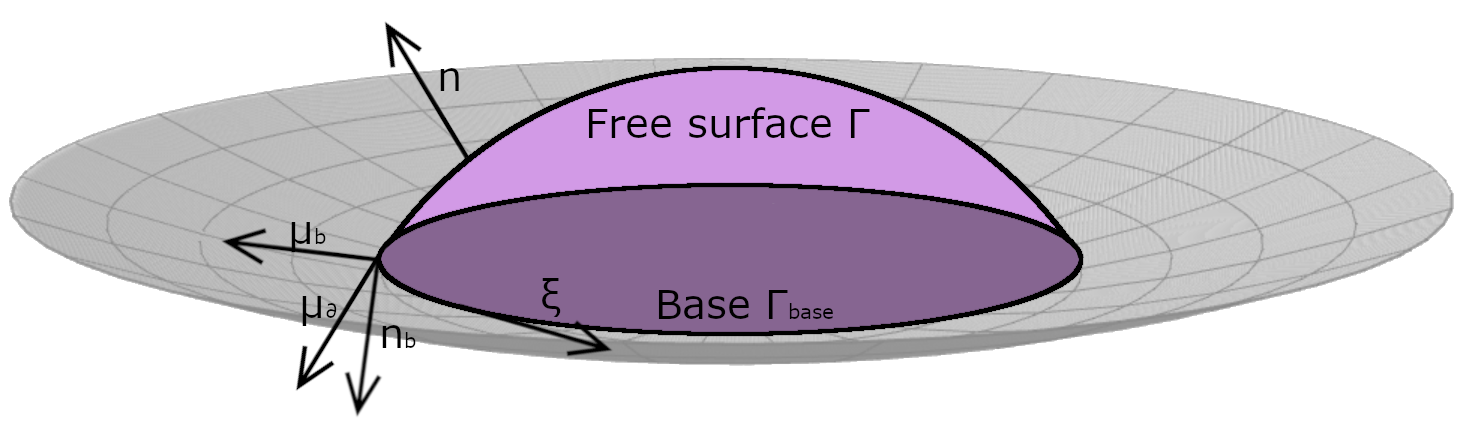}
    \caption{3D configuration}\label{domain_tube}
  \end{subfigure}
  \caption{Geometrical configurations.}
  \label{fig:domain_2D}
\end{figure}

Early finite element schemes for surface diffusion and related fourth-order
flows were developed for graphs and parametric surfaces by B\"ansch, Morin,
and Nochetto \cite{BMN2004,BMN2005}.  The parametric approach of Barrett,
Garcke, and N\"urnberg
\cite{BGN2007fourth,BGN2008hypersurfaces,BGN2020handbook}, commonly referred
to as the BGN method, leaves the tangential mesh velocity free.  This
freedom accommodates moving junctions \cite{BGN2010clusters,BGNZ2023networks}
and usually gives well-distributed meshes without remeshing.  Continuous-in-time, spatially discrete variants of the BGN approach are structure-preserving, in the sense that they decrease the surface energy while conserving the enclosed volume. The fully discrete BGN scheme also dissipates surface energy; however, it does not preserve the enclosed volume. A structure-preserving fully discrete BGN discretization for closed curves and surfaces was introduced by Bao and Zhao \cite{BaoZhao2021}. Extensions to curve networks and surface clusters were developed in \cite{BGNZ2023networks}, and an axisymmetric formulation was presented in \cite{BGNZ2022volume}. Structure-preserving schemes based on Lagrange multipliers were proposed in \cite{GaoLi2025MDR,GarckeJiangSuZhang2025}.

For solid-state dewetting on a flat substrate, parametric methods were
developed in two dimensions in \cite{BaoJiangWangZhao2017} and in three
dimensions in \cite{ZhaoJiangBao2020}; energy-stable variants were obtained
in \cite{ZhaoJiangBao2021,BaoZhao2023ssd3d}.  BGN methods also handle triple
junctions of curve networks and surface clusters, interfaces that meet a fixed
external boundary, and thermal grooving
\cite{BGN2007combined,BGN2008plane,BGN2010clusters,BGN2010grooving}.
Orientation-dependent surface energies have been treated by energy-stable
parametric formulations
\cite{BGN2008plane,BGN2008aniso,LiBao2021}, their symmetrized refinements
\cite{BaoJiangLi2023,BaoLi2023sym3d,BaoLi2024}, and diffuse-interface methods
\cite{GarckeKnopfNurnbergZhao2023}.

A curved substrate adds a geometric difficulty.  When the contact line moves, the areas of \(\Gamma_{\rm in}\) and \(\Gamma_{\rm out}\) change,
and on a curved substrate this change depends on where the contact line lies.
The curvature of the substrate also changes the dynamics: a small solid island migrates
along the curvature gradient of the substrate
\cite{ZhaoJiangWangSrolovitzQianBao2024}.  For general prescribed Young angles, existing parametric methods for
curved substrates treat curves: a curve attached to a prescribed substrate
curve \cite{BaoLiZhao2024curved}, or the generating curve of an axisymmetric
surface \cite{Zhang2025axisym}.  {{For surfaces in \(\mathbb R^3\) with non-right Young angles on general
curved substrates, we are not aware of a parametric method with a proved fully
discrete energy law.}}

We consider substrates given as the zero level set of a function \(\phi\).
The main idea of this paper is to mesh the substrate as well as the film
surface, so that the energies of all interfaces enter the discrete energy.
The substrate vertices move with the contact line, and at each of them the
discrete velocity is tangent to the substrate.  This gives two kinds of
schemes.  {The single-patch scheme glues the film surface to one
substrate patch, \(\Gamma_{\rm in}\) if \(\gamma_{\rm in}\ge\gamma_{\rm out}\)
and \(\Gamma_{\rm out}\) if \(\gamma_{\rm in}\le\gamma_{\rm out}\), and works
with this single surface, which is closed in the first case.  The full-substrate
scheme, on the other hand, meshes both patches, so the contact line becomes a triple junction of
three meshed surfaces.}  All these schemes are linear and
uniquely solvable, and applying the area inequality of
\cite{BGN2007fourth,BGN2008hypersurfaces} (see also \cite{Bansch2001}) on
every meshed surface shows that they are unconditionally energy stable.
Since the attachment constraint only makes the velocity of each substrate vertex
tangent to the substrate, the vertices may leave a curved substrate. Upon utilizing a degree of freedom in the choice of
discrete substrate normals, we are able to extend these linear schemes to a mildly nonlinear variant, which retains all
the desirable properties and, in addition, guarantees that the substrate vertices stay exactly on the substrate.
Extending the full-substrate scheme to anisotropic film and substrate energies is straightforward.
 To our knowledge, anisotropic substrate energies have not been treated by
parametric methods for curved substrates.

Our framework leaves the tangential motion of the vertices free.  Possible discrete tangential motions are the original BGN
tangential motion, or the natural analogue when the substrate normals are taken from \(\phi\), which we call
the level-set BGN motion, or the dual formulation \cite{GaoLiTang2026dual} of the
minimal-deformation-rate (MDR) tangential motion
\cite{HuLi2022MDR,GaoGarckeLiTang2025bgnmdr}; see also \cite{ZhangAndrewsFarrell2026,LiuZhao2026}.
In each case the obtained schemes are unconditionally energy stable and consistent approximations
of the considered flow.

Our main contributions are as follows.
\begin{enumerate}[label={\rm(\arabic*)},leftmargin=2.4em,itemsep=2pt,topsep=3pt]
  \item Parametric schemes that discretize the substrate together with the
  film surface: single-patch schemes for
  \(\gamma_{\rm in}\ge\gamma_{\rm out}\) and for
  \(\gamma_{\rm in}\le\gamma_{\rm out}\), and a full-substrate scheme for all
  densities.  They are linear, uniquely solvable, and unconditionally energy
  stable (Section~\ref{sec:schemes}).
  \item Mildly nonlinear extensions of the above schemes that keep the
  substrate vertices attached to the true substrate. As a solution method
  we propose a Picard iteration that is energy stable at every iterate
  (Section~\ref{sec:iterative}).
  \item An extension of the full-substrate scheme to anisotropic film and
  substrate energies, which is unconditionally energy stable
  (Section~\ref{sec:anisotropic}).
  \item A dual-MDR tangential motion that maintains the energy law and the
  consistency with the underlying flow.
  (Section~\ref{sec:dual-mdr}).
\end{enumerate}

Section~\ref{sec:continuous} presents the model and its weak formulations.
Section~\ref{sec:numerics} reports experiments on curves and surfaces,
including spherical, egg-carton, and sine-shaped substrates and an
anisotropic substrate energy.
Section~\ref{sec:conclusion} concludes with some final remarks.

\section{The continuous model and equivalent energy formulations}
\label{sec:continuous}

\subsection{Governing equations}
\label{subsec:problem}

Let $\Gamma(t)\subset\mathbb{R}^3$ be the free surface of a thin
film evolving on a curved substrate $\Gamma_{\rm base}$, with closed contact
lines forming $\partial\Gamma(t)=\Gamma(t)\cap\Gamma_{\rm base}$.
The constant positive surface-energy densities $\gamma_0$,
$\gamma_{\rm in}$, and $\gamma_{\rm out}$ correspond to the film/vapour,
film/substrate, and vapour/substrate interfaces, respectively.
During solid-state dewetting the morphology of $\Gamma(t)$ evolves by surface diffusion,
\begin{align}\label{eq:sd}
    v\cdot n=\beta\,\Delta_\Gamma(\gamma_0H) \qquad\text{on }\Gamma.
\end{align}
Here $v$ is the surface velocity, $v\cdot n$ is its normal
component, $n$ is the unit normal pointing out of the film, and $H$ is the mean
curvature; $\nabla_\Gamma$, $\nabla_\Gamma\cdot$, and $\Delta_\Gamma=\nabla_\Gamma\cdot\nabla_\Gamma$ are the surface gradient, divergence, and Laplace--Beltrami operator on $\Gamma$. The surface mobility constant $\beta>0$ is related to the surface diffusion constant \cite{Mullins1957}. Let $n_b$ be the unit normal of $\Gamma_{\rm base}$ pointing into
the substrate. The evolution is supplemented by three conditions along the
contact lines $\partial\Gamma$:
\begin{subequations}\label{eq:bc}
\begin{align}
    v\cdot n_b &= 0, \label{eq:bc-attach}\\
    \mu_{\scriptscriptstyle\partial}\cdot\mu_b &= \frac{\gamma_{\rm out}-\gamma_{\rm in}}{\gamma_0}, \label{eq:bc-young}\\
    \bigl(\mu_{\scriptscriptstyle\partial}\cdot\nabla_{\Gamma}H\bigr)\big|_{\partial\Gamma} &= 0. \label{eq:bc-flux}
\end{align}
\end{subequations}
The outward unit conormal
$\mu_{\scriptscriptstyle\partial}$ of $\Gamma$ along $\partial\Gamma$ is
tangent to $\Gamma$, normal to $\partial\Gamma$, and points away from the film,
while $\mu_b$ is the unit conormal of $\Gamma_{\rm in}$ along $\partial\Gamma$
within $\Gamma_{\rm base}$, pointing out of $\Gamma_{\rm in}$ into
$\Gamma_{\rm out}$. Condition \eqref{eq:bc-attach} keeps the contact lines on
the substrate, \eqref{eq:bc-young} is Young's law, and \eqref{eq:bc-flux} states
that no mass flows across them.

\subsection{The total free energy and Young's relation}
\label{subsec:energy}

We assume that $\Gamma(t)$ and $\Gamma_{\rm base}$ are
connected, that each component of $\partial\Gamma(t)$ encloses a patch of
$\Gamma_{\rm base}$, that these patches have pairwise disjoint closures, and
that the contact lines stay in the interior of the substrate,
\begin{align}\label{eq:cl-interior}
\partial\Gamma(t)\cap\partial\Gamma_{\rm base}=\emptyset .
\end{align}
Let
$\Gamma_{\rm in}(t)\subset\Gamma_{\rm base}$ be the union of the enclosed
patches (on a closed substrate, the patches covered by the film) and
$\Gamma_{\rm out}(t):=\Gamma_{\rm base}\setminus\Gamma_{\rm in}(t)\neq\emptyset$.  Then
\begin{align}\label{eq:area-const}
|\Gamma_{\rm in}(t)|+|\Gamma_{\rm out}(t)|=|\Gamma_{\rm base}|=\text{constant}.
\end{align}
The total free energy is
\begin{align}\label{eq:energy-W}
W(t):=\gamma_0|\Gamma(t)|+\gamma_{\rm in}|\Gamma_{\rm in}(t)|+\gamma_{\rm out}|\Gamma_{\rm out}(t)|.
\end{align}
The equilibrium contact angle $\theta\in(0,\pi)$ is related
to the densities by
\begin{align}\label{eq:young-energy}
\frac{\gamma_{\rm out}-\gamma_{\rm in}}{\gamma_0}=\cos\theta=\mu_{\scriptscriptstyle\partial}\cdot\mu_b,
\end{align}
so \eqref{eq:bc-young} states that the contact angle equals \(\theta\).

\subsection{Three equivalent energy reformulations}
\label{subsec:three-energies}

For the single-patch schemes, we eliminate one substrate
contribution from $W$ and discretize the film surface together with the
remaining substrate patch.  The glued surface
$\Gamma\cup\Gamma_{\rm in}$ is closed, whereas
$\Gamma\cup\Gamma_{\rm out}$ has only the fixed outer rim
$\partial\Gamma_{\rm base}$ as its boundary, which is empty if
$\Gamma_{\rm base}$ is closed.  Using \eqref{eq:area-const} in
\eqref{eq:energy-W} gives two equivalent forms,
\begin{align}
W(t) &= \underbrace{\gamma_0|\Gamma(t)|+(\gamma_{\rm in}-\gamma_{\rm out})|\Gamma_{\rm in}(t)|}_{W_1(t)}+\gamma_{\rm out}|\Gamma_{\rm base}|, \label{eq:W1-equiv}\\
W(t) &= \underbrace{\gamma_0|\Gamma(t)|+(\gamma_{\rm out}-\gamma_{\rm in})|\Gamma_{\rm out}(t)|}_{W_2(t)}+\gamma_{\rm in}|\Gamma_{\rm base}|. \label{eq:W2-equiv}
\end{align}
Surface diffusion \eqref{eq:sd} with the boundary conditions \eqref{eq:bc}
gives
\begin{align}\label{eq:diss}
\frac{\diff}{\diff t}W(t)=-\beta\,\|\nabla_{\Gamma(t)}(\gamma_0H)\|_{L^2(\Gamma(t))}^2\le 0.
\end{align}
Since $W$, $W_1$, and $W_2$ differ only by additive constants, \(W_1\) and
\(W_2\) decrease at the same rate.
For energy-stable single-patch discretizations, we use
$W_1$ when $\gamma_{\rm in}\ge\gamma_{\rm out}$ and $W_2$ when
$\gamma_{\rm in}\le\gamma_{\rm out}$, so the coefficients
of the reduced energies are nonnegative.  The full-substrate scheme instead
works directly with $W$.
\subsection{Variational identities on the free surface and substrate patches}
\label{subsec:variational}

Sobolev spaces on unions of patches, such as
\(H^1(\Gamma\cup\Gamma_{\rm in})^3\), are understood patchwise, with a single
trace on each shared interface, and
\(H^1_{0,\partial\Gamma_{\rm base}}\) denotes functions that vanish on the
fixed outer rim \(\partial\Gamma_{\rm base}\).  Since
\(\mu_{\scriptscriptstyle\partial}\) is a unit vector in the plane spanned by
\(\mu_b\) and \(n_b\), Young's law \eqref{eq:bc-young} gives
\(\mu_{\scriptscriptstyle\partial}=\cos\theta\,\mu_b+\sin\theta\,n_b\).
Integrating \(-Hn=\Delta_\Gamma\id\) by parts on \(\Gamma\), where \(\id\) is
the identity map, and the same identity on the two substrate patches, whose
outward conormals along \(\partial\Gamma\) are \(\mu_b\) and \(-\mu_b\), gives
\begin{align}\label{foundation}
\int_{\Gamma} Hn\cdot w
=
\int_{\Gamma}\nabla_{\Gamma}{\rm id}\cdot\nabla_{\Gamma}w
-\cos\theta\int_{\partial\Gamma}\mu_b\cdot w
-\sin\theta\int_{\partial\Gamma}n_b\cdot w ,
\end{align}
\begin{subequations}\label{foundation_substrate}
\begin{align}
\int_{\Gamma_{\rm in}} H_{\rm in} n_b\cdot w
&=
\int_{\Gamma_{\rm in}}
\nabla_{\Gamma_{\rm in}}{\rm id}\cdot\nabla_{\Gamma_{\rm in}}w
-\int_{\partial\Gamma}\mu_b\cdot w,
\label{foundation_in}
\\
\int_{\Gamma_{\rm out}} H_{\rm out} n_b\cdot w
&=
\int_{\Gamma_{\rm out}}
\nabla_{\Gamma_{\rm out}}{\rm id}\cdot\nabla_{\Gamma_{\rm out}}w
+\int_{\partial\Gamma}\mu_b\cdot w,
\label{foundation_out}
\end{align}
\end{subequations}
with \(H_{\rm in}\), \(H_{\rm out}\) the mean curvatures of the patches and
\(w=0\) on \(\partial\Gamma_{\rm base}\) in \eqref{foundation_out}.

Adding \eqref{foundation} times \(\gamma_0\) and \eqref{foundation_in} times
\(\gamma_{\rm in}-\gamma_{\rm out}=-\gamma_0\cos\theta\) cancels the
\(\mu_b\)-terms.  The remaining substrate-curvature and contact-line terms are
parallel to \(n_b\), so they vanish on
\(\mathcal V_{\Gamma\cup\Gamma_{\rm in}}:=\{w\in H^1(\Gamma\cup\Gamma_{\rm in})^3:
w\cdot n_b=0\text{ a.e.\ on }\Gamma_{\rm in}\}\), whose constraint also holds
for the trace on \(\partial\Gamma\).  Hence, for all
\(w\in\mathcal V_{\Gamma\cup\Gamma_{\rm in}}\),
\begin{align}\label{eq:var-W1}
\gamma_0\int_{\Gamma} Hn\cdot w
=\gamma_0\int_{\Gamma}\nabla_{\Gamma}{\rm id}\cdot\nabla_{\Gamma}w
+(\gamma_{\rm in}-\gamma_{\rm out})\int_{\Gamma_{\rm in}}
\nabla_{\Gamma_{\rm in}}{\rm id}\cdot\nabla_{\Gamma_{\rm in}}w .
\end{align}
In the same way, \eqref{foundation} and \eqref{foundation_out} with weights
\(\gamma_0\) and \(\gamma_{\rm out}-\gamma_{\rm in}\) give, for all \(w\) in
\(\mathcal V_{\Gamma\cup\Gamma_{\rm out}}:=\{w\in
H^1_{0,\partial\Gamma_{\rm base}}(\Gamma\cup\Gamma_{\rm out})^3:
w\cdot n_b=0\text{ a.e.\ on }\Gamma_{\rm out}\}\),
\begin{align}\label{eq:var-W2}
\gamma_0\int_{\Gamma} Hn\cdot w
=\gamma_0\int_{\Gamma}\nabla_{\Gamma}{\rm id}\cdot\nabla_{\Gamma}w
+(\gamma_{\rm out}-\gamma_{\rm in})\int_{\Gamma_{\rm out}}
\nabla_{\Gamma_{\rm out}}{\rm id}\cdot\nabla_{\Gamma_{\rm out}}w ,
\end{align}
and all three identities with weights \(\gamma_0\), \(\gamma_{\rm in}\),
\(\gamma_{\rm out}\) give, on the full-substrate space
\begin{align}\label{eq:space-W}
  \mathcal V_{\Gamma\cup\Gamma_{\rm base}}
  :=
  \left\{
  w\in H^1_{0,\partial\Gamma_{\rm base}}
  (\Gamma\cup\Gamma_{\rm base})^3:\;
  w\cdot n_b=0 \ \text{a.e.\ on } \Gamma_{\rm base}
  \right\},
\end{align}
\begin{multline}\label{eq:var-W}
\gamma_0\int_{\Gamma} Hn\cdot w
=\gamma_0\int_{\Gamma}\nabla_{\Gamma}{\rm id}\cdot\nabla_{\Gamma}w
+\gamma_{\rm in}\int_{\Gamma_{\rm in}}
\nabla_{\Gamma_{\rm in}}{\rm id}\cdot\nabla_{\Gamma_{\rm in}}w\\
+\gamma_{\rm out}\int_{\Gamma_{\rm out}}
\nabla_{\Gamma_{\rm out}}{\rm id}\cdot\nabla_{\Gamma_{\rm out}}w
\qquad
\forall w\in\mathcal V_{\Gamma\cup\Gamma_{\rm base}} .
\end{multline}

Conversely, for smooth surfaces each identity implies \(-Hn=\Delta_\Gamma\id\) and Young's law.
For example, by \eqref{foundation_substrate} and \(w\cdot n_b=0\) on $\Gamma_{\rm base}$, the two substrate terms on the right hand side of \eqref{eq:var-W} reduce to \((\gamma_{\rm in}-\gamma_{\rm out})\int_{\partial\Gamma}\mu_b\cdot w\). Then integrating the free-surface term by parts gives
\begin{align}\label{eq:geo-consistency-W}
\gamma_0\int_\Gamma (Hn+\Delta_\Gamma\id)\cdot w
=\int_{\partial\Gamma}\bigl[\gamma_0\mu_{\scriptscriptstyle\partial}
+(\gamma_{\rm in}-\gamma_{\rm out})\mu_b\bigr]\cdot w
\qquad \forall w\in\mathcal V_{\Gamma\cup\Gamma_{\rm base}} .
\end{align}
Test functions supported in the interior of \(\Gamma\) give
\(-Hn=\Delta_\Gamma\id\), while choosing \(w=\zeta\mu_b\) on \(\partial\Gamma\), for
\(\zeta\) arbitrary, gives Young's law \eqref{eq:bc-young}.

\section{The substrate-coupled BGN framework}
\label{sec:schemes}

The identities \eqref{eq:var-W1}, \eqref{eq:var-W2}, and \eqref{eq:var-W}
lead to three linear energy-stable schemes of BGN type.  We call the framework
substrate-coupled because the substrate mesh is part of the scheme.  The
single-patch schemes couple the free surface $\Gamma_h$ to one substrate patch,
$\Gamma_{{\rm in},h}$ if $\gamma_{\rm in}\ge\gamma_{\rm out}$ and
$\Gamma_{{\rm out},h}$ otherwise.  The full-substrate scheme couples it to
both.  Each time step solves one linear system for a velocity on the coupled
meshes and a mean curvature on $\Gamma_h$.

\subsection{Finite element preliminaries}
\label{subsec:fem}

Let $\Gamma^0\subset\mathbb{R}^3$ be a smooth initial surface, approximated by a triangulated surface $\Gamma_h^0$ with element set $\mathcal K_h(\Gamma_h^0)$. The interval $[0,T]$ is partitioned uniformly as $t_m=m\tau$, $m=0,1,\dots,N$, with step $\tau>0$. At time level $t_m$ the surface is $\Gamma_h^m=\bigcup_{K\in\mathcal K_h(\Gamma_h^m)}K$.
Alongside $\Gamma_h^m$ we discretize the substrate patches by triangulated surfaces $\Gamma_{{\rm in},h}^m$ and $\Gamma_{{\rm out},h}^m$. For compact notation we set
\[
\Gamma_{{\rm base},h}^m:=\Gamma_{{\rm in},h}^m\cup\Gamma_{{\rm out},h}^m,\qquad
{\Gamma_{\cup\star,h}^m:=\Gamma_h^m\cup\Gamma_{\star,h}^m\quad\text{for }\star\in\{{\rm in},{\rm out},{\rm base}\}.}
\]

For a triangulated surface $\Sigma_h$ among the above, let
\[
S_h(\Sigma_h):=\bigl\{v_h\in C^0(\Sigma_h):\ v_h|_K\in\mathbb P_1(K)\ \text{for all }K\in\mathcal K_h(\Sigma_h)\bigr\},
\]
with vector-valued counterpart $(S_h(\Sigma_h))^3$. The vertices of each $K=\triangle(q_{K,0},q_{K,1},q_{K,2})$ are ordered counterclockwise as seen from the exterior of the enclosed region. We use the mass-lumped inner product
\begin{align}\label{mass-lump-defn}
   \int_{\Sigma_h}^hu\cdot w:=\sum_{K\in\mathcal K_h(\Sigma_h)}\frac{|K|}{3}\sum_{j=0}^2(u\cdot w)|_K(q_{K,j}),
\end{align}
used with the same formula for scalar-valued factors,
and the mass-lumped Lagrange interpolant $I_{h,\Sigma_h}f\in S_h(\Sigma_h)$ of a possibly elementwise discontinuous $f$, defined nodally by
\begin{align}\label{eq:lumped-int}
(I_{h,\Sigma_h}f)(q)=\frac{\sum_{K\ni q}f|_K(q)\,|K|}{\sum_{K\ni q}|K|}\qquad\text{for every vertex }q\in\Sigma_h.
\end{align}

Let $n_h^m$ be the piecewise constant unit normal of the free surface
$\Gamma_h^m$ and let
\[\hat n_h^m:=I_{h,\Gamma_h^m}n_h^m\] be its mass-lumped, area-weighted nodal
normal.

The attachment constraint below needs a discrete substrate normal $\bar n_{\star,h}^m$ at each substrate vertex, for $\star\in\{{\rm in},{\rm out},{\rm base}\}$.
When $\Gamma_{\rm base}$ is the zero
level set of a smooth $\phi:\mathbb R^3\to\mathbb R$ with $|\nabla\phi|>0$ in a
{neighbourhood} $\mathcal U$ of $\Gamma_{\rm base}$, we prescribe
\begin{align}\label{eq:level-set-normal}
\bar n_{{\rm base},h}^m(q):=\frac{\nabla\phi(q)}{|\nabla\phi(q)|}\qquad\text{for every vertex }q\in\Gamma_{{\rm base},h}^m,
\end{align}
and similarly for $\bar n_{{\rm in},h}^m$ and
$\bar n_{{\rm out},h}^m$. The vertices of $\Gamma_{{\rm base},h}^m$ need not lie on
$\Gamma_{\rm base}$.  Formula \eqref{eq:level-set-normal} only needs them to
stay in $\mathcal U$, which we assume throughout.  Each value is the unit
normal to the level set of $\phi$ through its vertex,
which agrees with the substrate normal
when the vertex lies on $\Gamma_{\rm base}$.

Let $\mathcal N(\Sigma_h)$ denote the vertex set of a triangulated surface $\Sigma_h$. For $\star\in\{{\rm in},{\rm out},{\rm base}\}$ set
\[
\mathcal Q_{\star,h}^m:=\mathcal N(\Gamma_{\star,h}^m).
\]
Each contact-line vertex occurs only once in $\mathcal Q_{{\rm base},h}^m$, although it is incident to both substrate patches. The admissible velocity space is
\begin{align}\label{eq:subspace}
\mathcal V_{\cup\star,h}^m:=\bigl\{\chi_h\in(S_h(\Gamma_{\cup\star,h}^m))^3:\ \chi_h(q)\cdot\bar n_{\star,h}^m(q)=0\ \text{for all }q\in\mathcal Q_{\star,h}^m\bigr\}.
\end{align}
This constraint approximates the attachment condition \eqref{eq:bc-attach} at
the substrate vertices and leaves the tangential motion free.
We also pin the velocity
on the fixed outer boundary $\partial\Gamma_{{\rm base},h}^m$, whose vertices
we call pinned: for $\star\in\{{\rm in},{\rm out},{\rm base}\}$,
\begin{align}\label{eq:outer-subspace}
\mathring{\mathcal V}_{\cup\star,h}^m:=\bigl\{\chi_h\in\mathcal V_{\cup\star,h}^m:\ \chi_h(q)=0\ \text{for all vertices }q\in\partial\Gamma_{{\rm base},h}^m\bigr\}.
\end{align}
For $\star={\rm in}$, the inner patch contains no vertex of
$\partial\Gamma_{{\rm base},h}^m$, recall \eqref{eq:cl-interior}, so the added
condition is empty and
$\mathring{\mathcal V}_{\cup{\rm in},h}^m=\mathcal V_{\cup{\rm in},h}^m$.

\begin{remark}
\label{rem:normal-consistency}
One could also use a substrate normal computed from the substrate mesh, such as
the mass-lumped vertex normal \eqref{eq:lumped-int} of
$\Gamma_{{\rm base},h}^m$. This would give rise to BGN tangential motion for the
substrate vertices. But in situations where $\Gamma_{\rm base}$ is given by a level set
function, we prefer the choice \eqref{eq:level-set-normal}.
\end{remark}

\subsection{Single-patch schemes}
\label{subsec:single-patch}

The single-patch schemes discretize \eqref{eq:var-W1} or
\eqref{eq:var-W2}.  Set
\begin{align}\label{eq:sigma}
(\star,\sigma):=
\begin{cases}
({\rm in},\,\gamma_{\rm in}-\gamma_{\rm out}), & \gamma_{\rm in}\ge\gamma_{\rm out},\\[2pt]
({\rm out},\,\gamma_{\rm out}-\gamma_{\rm in}), & \gamma_{\rm in}\le\gamma_{\rm out},
\end{cases}
\end{align}
so that $\sigma\ge0$; for $\gamma_{\rm in}=\gamma_{\rm out}$
either choice may be used.  Given $\Gamma_h^m$ and
$\Gamma_{\star,h}^m$, find
$(v_h^{m+1},H_h^{m+1})\in\mathring{\mathcal V}_{\cup\star,h}^m\times S_h(\Gamma_h^m)$
such that, with $X_h^{m+1}:=\id+\tau v_h^{m+1}$ and for all
$(\varphi_h,w_h)\in S_h(\Gamma_h^m)\times\mathring{\mathcal V}_{\cup\star,h}^m$,
\begin{subequations}\label{eq:scheme-single}
\begin{align}
&\int_{\Gamma_h^m}^h(v_h^{m+1}\cdot n_h^m)\,\varphi_h
+\beta\int_{\Gamma_h^m}\nabla_{\Gamma_h^m}(\gamma_0H_h^{m+1})\cdot\nabla_{\Gamma_h^m}\varphi_h
=0, \label{eq:scheme-single-flow}\\
&\gamma_0\int_{\Gamma_h^m}^hH_h^{m+1}\,n_h^m\cdot w_h
=\gamma_0\int_{\Gamma_h^m}\nabla_{\Gamma_h^m}X_h^{m+1}\cdot\nabla_{\Gamma_h^m}w_h
\notag\\
&\quad+\sigma\int_{\Gamma_{\star,h}^m}
\nabla_{\Gamma_{\star,h}^m}X_h^{m+1}\cdot
\nabla_{\Gamma_{\star,h}^m}w_h .
\label{eq:scheme-single-curv}
\end{align}
\end{subequations}
The new surfaces are $\Gamma_h^{m+1}:=X_h^{m+1}(\Gamma_h^m)$ and
$\Gamma_{\star,h}^{m+1}:=X_h^{m+1}(\Gamma_{\star,h}^m)$.
Equation \eqref{eq:scheme-single-curv}, which is the {discrete version of \eqref{eq:var-W1} or \eqref{eq:var-W2}},
induces a tangential motion.
When the attachment condition is enforced through the level set normal as in
\eqref{eq:subspace}, we call this the level-set BGN tangential motion.  The scheme dissipates
\begin{align}\label{eq:single-energy}
W_{\star,h}^m=\gamma_0|\Gamma_h^m|+\sigma\,|\Gamma_{\star,h}^m| ,
\end{align}
the discrete counterpart of $W_1$ in \eqref{eq:W1-equiv} for
$\star={\rm in}$, and of $W_2$ in \eqref{eq:W2-equiv} for $\star={\rm out}$.

\begin{remark}
\label{rem:equal-densities}
If $\gamma_{\rm in}=\gamma_{\rm out}$, the Young angle is $90^\circ$ and
$\sigma=0$, so the substrate integral in \eqref{eq:scheme-single-curv}
vanishes.  The substrate patch then need not be meshed: \eqref{eq:scheme-single}
becomes a scheme on $\Gamma_h^m$ alone, with velocities in
\[
\bigl\{\chi_h\in(S_h(\Gamma_h^m))^3:\ \chi_h(q)\cdot\bar n_{{\rm base},h}^m(q)=0
\ \text{for all }q\in\partial\Gamma_h^m\bigr\},
\]
so the attachment constraint acts only at the contact-line vertices,
compare also with \cite[(4.13)]{BGN2010clusters}.
\end{remark}

\subsection{Full-substrate scheme}
\label{subsec:full-substrate}

The full-substrate scheme discretizes \eqref{eq:var-W}.  Both substrate
patches are meshed, so we do not need to know which density is larger.  Given
\(\Gamma_h^m\), \(\Gamma_{{\rm in},h}^m\), and \(\Gamma_{{\rm out},h}^m\), find
\((v_h^{m+1},H_h^{m+1})\in
\mathring{\mathcal V}_{\cup{\rm base},h}^m\times S_h(\Gamma_h^m)\) such that,
with \(X_h^{m+1}:=\id+\tau v_h^{m+1}\) and for all
\((\varphi_h,w_h)\in S_h(\Gamma_h^m)\times
\mathring{\mathcal V}_{\cup{\rm base},h}^m\),
\begin{subequations}\label{eq:scheme-unified}
\begin{align}
&\int_{\Gamma_h^m}^h(v_h^{m+1}\cdot n_h^m)\,\varphi_h
+\beta\int_{\Gamma_h^m}\nabla_{\Gamma_h^m}(\gamma_0H_h^{m+1})\cdot\nabla_{\Gamma_h^m}\varphi_h
=0, \label{eq:scheme-unified-flow}\\
&\gamma_0\int_{\Gamma_h^m}^hH_h^{m+1}\,n_h^m\cdot w_h
=\gamma_0\int_{\Gamma_h^m}\nabla_{\Gamma_h^m}X_h^{m+1}\cdot
\nabla_{\Gamma_h^m}w_h
\notag\\
&\quad+\gamma_{\rm in}\int_{\Gamma_{{\rm in},h}^m}
\nabla_{\Gamma_{{\rm in},h}^m}X_h^{m+1}\cdot
\nabla_{\Gamma_{{\rm in},h}^m}w_h+\gamma_{\rm out}\int_{\Gamma_{{\rm out},h}^m}
\nabla_{\Gamma_{{\rm out},h}^m}X_h^{m+1}\cdot
\nabla_{\Gamma_{{\rm out},h}^m}w_h .
\label{eq:scheme-unified-curv}
\end{align}
\end{subequations}
The new surfaces are $\Gamma_h^{m+1}:=X_h^{m+1}(\Gamma_h^m)$,
$\Gamma_{{\rm in},h}^{m+1}:=X_h^{m+1}(\Gamma_{{\rm in},h}^m)$ and
$\Gamma_{{\rm out},h}^{m+1}:=X_h^{m+1}(\Gamma_{{\rm out},h}^m)$.
Similarly to \eqref{eq:scheme-single-curv}, also \eqref{eq:scheme-unified-curv}
induces a tangential motion on the discrete free surface and on the substrate
approximations.
Moreover, the scheme dissipates
\begin{align}\label{eq:full-energy}
W_h^m=\gamma_0|\Gamma_h^m|+\gamma_{\rm in}|\Gamma_{{\rm in},h}^m|+\gamma_{\rm out}|\Gamma_{{\rm out},h}^m|.
\end{align}

\subsection{Well-posedness}
\label{subsec:wellposed}

We first show that the single-patch scheme \eqref{eq:scheme-single} has a
unique solution.  The same argument covers the full-substrate scheme
\eqref{eq:scheme-unified}.

\begin{theorem}[Well-posedness of the single-patch schemes]
\label{thm:wellposed-neg}
Let $(\star,\sigma)$ be as in \eqref{eq:sigma} with $\sigma>0$, and let
$\Gamma_h^m$ and $\Gamma_{\star,h}^m$ satisfy:
\begin{enumerate}[label={\rm(A\arabic*)}]
    \item the triangulations of $\Gamma_h^m$ and $\Gamma_{\star,h}^m$ are regular with no degenerate elements, the contact line $\partial\Gamma_h^m=\Gamma_h^m\cap\Gamma_{\star,h}^m$ is contained in $\partial\Gamma_{\star,h}^m$ and is disjoint from the fixed rim $\partial\Gamma_{{\rm base},h}^m$, and the two triangulations induce the same partition of the contact line; moreover, $\Gamma_h^m$ and the glued union $\Gamma_h^m\cup\Gamma_{\star,h}^m$ are connected;
    \item for every vertex $q\in\Gamma_h^m$ the area-weighted nodal normal $\hat n_h^m(q)$ is well defined and nonzero, and these normals span $\mathbb R^3$;
    \item for every contact-line vertex $q\in\partial\Gamma_h^m$ the vectors $\hat n_h^m(q)$ and $\bar n_{\star,h}^m(q)$ are linearly independent.
\end{enumerate}
Then the scheme \eqref{eq:scheme-single} has a unique solution
$(v_h^{m+1},H_h^{m+1})\in\mathring{\mathcal V}_{\cup\star,h}^m\times S_h(\Gamma_h^m)$.
\end{theorem}

\begin{proof}
The trial and test spaces have the same finite dimension, so it suffices to
show that the homogeneous system has only the trivial solution.  Let
$(v_h^{m+1},H_h^{m+1})$ satisfy, for all
\((\varphi_h,w_h)\in S_h(\Gamma_h^m)\times\mathring{\mathcal V}_{\cup\star,h}^m\),
\begin{subequations}\label{eq:homo-neg}
\begin{align}
\int_{\Gamma_h^m}^h(v_h^{m+1}\cdot n_h^m)\,\varphi_h
&+\beta\int_{\Gamma_h^m}\nabla_{\Gamma_h^m}(\gamma_0H_h^{m+1})\cdot\nabla_{\Gamma_h^m}\varphi_h=0, \label{eq:homo-neg-flow}\\
\gamma_0\int_{\Gamma_h^m}^hH_h^{m+1}\,n_h^m\cdot w_h
&=\tau\gamma_0\int_{\Gamma_h^m}\nabla_{\Gamma_h^m}v_h^{m+1}\cdot\nabla_{\Gamma_h^m}w_h \notag\\
&\quad+\tau\sigma\int_{\Gamma_{\star,h}^m}\nabla_{\Gamma_{\star,h}^m}v_h^{m+1}\cdot\nabla_{\Gamma_{\star,h}^m}w_h. \label{eq:homo-neg-curv}
\end{align}
\end{subequations}
Taking $w_h=v_h^{m+1}$ in \eqref{eq:homo-neg-curv} and
$\varphi_h=-\gamma_0H_h^{m+1}$ in \eqref{eq:homo-neg-flow} gives
\begin{align*}
&\beta\int_{\Gamma_h^m}|\nabla_{\Gamma_h^m}(\gamma_0H_h^{m+1})|^2
+\tau\gamma_0\int_{\Gamma_h^m}|\nabla_{\Gamma_h^m}v_h^{m+1}|^2+\tau\sigma\int_{\Gamma_{\star,h}^m}|\nabla_{\Gamma_{\star,h}^m}v_h^{m+1}|^2=0.
\end{align*}
Since $\sigma>0$, all terms are nonnegative.  Hence $H_h^{m+1}$ is constant
on $\Gamma_h^m$, and $v_h^{m+1}$ is constant on the glued union
$\Gamma_h^m\cup\Gamma_{\star,h}^m$, which is connected by {\rm(A1)}.  Since
$\nabla_{\Gamma_h^m}H_h^{m+1}=0$, the flow equation \eqref{eq:homo-neg-flow}
reduces to \(\int_{\Gamma_h^m}^h(v_h^{m+1}\cdot n_h^m)\varphi_h=0\) for
all \(\varphi_h\).  Testing with nodal basis functions and using
\eqref{eq:lumped-int} gives
\[
v_h^{m+1}(q)\cdot\hat n_h^m(q)=0\qquad\text{at every vertex }q\in\Gamma_h^m.
\]
Hence the constant value of $v_h^{m+1}$ is orthogonal to all
$\hat n_h^m(q)$, so it vanishes by {\rm(A2)}, and \(v_h^{m+1}\equiv0\).  The
right-hand side of \eqref{eq:homo-neg-curv} vanishes, so
\begin{align}\label{eq:reduced-neg}
\int_{\Gamma_h^m}^hH_h^{m+1}\,n_h^m\cdot w_h=0\qquad\forall\,w_h\in\mathring{\mathcal V}_{\cup\star,h}^m.
\end{align}
Let \(q\) be a vertex of \(\Gamma_h^m\) and let \(\varphi_q\) be its nodal
basis function on \(\Gamma_{\cup\star,h}^m\).  If \(q\notin\partial\Gamma_h^m\),
then \(w_h=\varphi_q\hat n_h^m(q)\) is admissible, and \eqref{eq:reduced-neg}
gives
\[
\omega_q^\Gamma\,H_h^{m+1}(q)\,|\hat n_h^m(q)|^2=0,
\]
where \(\omega_q^\Gamma>0\) is the lumped mass weight at \(q\).  Thus
\(H_h^{m+1}(q)=0\).

For a contact-line vertex \(q\in\partial\Gamma_h^m\), set
\[
a_q:=\hat n_h^m(q)-\frac{\hat n_h^m(q)\cdot\bar n_{\star,h}^m(q)}{|\bar n_{\star,h}^m(q)|^2}\,\bar n_{\star,h}^m(q).
\]
Then \(a_q\cdot\bar n_{\star,h}^m(q)=0\), so
\(\varphi_q a_q\in\mathring{\mathcal V}_{\cup\star,h}^m\), and {\rm(A3)} gives
\(a_q\neq0\).  Testing \eqref{eq:reduced-neg} with \(w_h=\varphi_q a_q\) gives
\[
\omega_q^\Gamma\,H_h^{m+1}(q)\,\hat n_h^m(q)\cdot a_q=0,
\qquad
\hat n_h^m(q)\cdot a_q=|a_q|^2>0,
\]
so again \(H_h^{m+1}(q)=0\).  Hence \(H_h^{m+1}\) vanishes at every vertex of
\(\Gamma_h^m\), that is, \(H_h^{m+1}\equiv0\).  Together with
\(v_h^{m+1}\equiv0\), this shows that \eqref{eq:scheme-single} has a unique
solution.
\end{proof}

\begin{theorem}[Well-posedness of the full-substrate scheme]
\label{thm:wellposed-unified}
Let $\Gamma_h^m$ and the glued substrate $\Gamma_{{\rm base},h}^m$ satisfy
{\rm(A1)}--{\rm(A3)} with $\Gamma_{{\rm base},h}^m$ in place of
$\Gamma_{\star,h}^m$, except that the contact line lies in the interior of the
substrate mesh rather than on its boundary.  Then the full-substrate scheme
\eqref{eq:scheme-unified} has a unique solution
$(v_h^{m+1},H_h^{m+1})\in\mathring{\mathcal V}_{\cup{\rm base},h}^m\times S_h(\Gamma_h^m)$.
\end{theorem}

\begin{proof}
{As for Theorem~\ref{thm:wellposed-neg}, with
\(\Gamma_{{\rm base},h}^m\) in place of \(\Gamma_{\star,h}^m\).  The
homogeneous form of \eqref{eq:scheme-unified-curv} has three stiffness terms
with positive weights \(\gamma_0\), \(\gamma_{\rm in}\), \(\gamma_{\rm out}\),
and at a contact-line vertex the test direction \(a_q\) is formed with
\(\bar n_{{\rm base},h}^m(q)\).}
\end{proof}

\subsection{Energy stability}
\label{subsec:energy-stable}

\begin{theorem}[Energy stability of the single-patch schemes]
\label{thm:energy-neg}
The scheme \eqref{eq:scheme-single} satisfies
\begin{align}\label{eq:area-decay-1}
W_{\star,h}^{m+1}+\tau\beta\|\nabla_{\Gamma_h^m}(\gamma_0H_h^{m+1})\|_{L^2(\Gamma_h^m)}^2\le W_{\star,h}^m,\qquad
W_{\star,h}^m=\gamma_0|\Gamma_h^m|+\sigma\,|\Gamma_{\star,h}^m| .
\end{align}
\end{theorem}

\begin{proof}
Choosing $\varphi_h=\gamma_0H_h^{m+1}$ in \eqref{eq:scheme-single-flow}
and
$w_h=v_h^{m+1}$ in \eqref{eq:scheme-single-curv} gives
\begin{align}\label{eq:neg-inter}
\begin{aligned}
&\gamma_0\int_{\Gamma_h^m}\nabla_{\Gamma_h^m}X_h^{m+1}\cdot\nabla_{\Gamma_h^m}v_h^{m+1}
+\sigma\int_{\Gamma_{\star,h}^m}\nabla_{\Gamma_{\star,h}^m}X_h^{m+1}\cdot\nabla_{\Gamma_{\star,h}^m}v_h^{m+1}\\
&\quad=-\beta\|\nabla_{\Gamma_h^m}(\gamma_0H_h^{m+1})\|_{L^2(\Gamma_h^m)}^2.
\end{aligned}
\end{align}
On any triangulated surface $\Sigma_h^m$ moved by $X_h^{m+1}=\id+\tau v_h^{m+1}$,
the area inequality (see \cite[(2.21)]{BGN2008hypersurfaces}, and
also \cite{Bansch2001}) bounds the area
change by the stiffness term,
\begin{align}\label{eq:area-ineq-neg}
\tau\int_{\Sigma_h^m}\nabla_{\Sigma_h^m}X_h^{m+1}\cdot\nabla_{\Sigma_h^m}v_h^{m+1}
\ge|\Sigma_h^{m+1}|-|\Sigma_h^m|.
\end{align}
Both stiffness terms in \eqref{eq:neg-inter} carry nonnegative coefficients,
$\gamma_0$ and $\sigma$, so multiplying \eqref{eq:neg-inter} by $\tau$ and
applying \eqref{eq:area-ineq-neg} on $\Gamma_h^m$ and on $\Gamma_{\star,h}^m$
gives
\[
W_{\star,h}^{m+1}-W_{\star,h}^m
=\gamma_0\bigl(|\Gamma_h^{m+1}|-|\Gamma_h^m|\bigr)+\sigma\bigl(|\Gamma_{\star,h}^{m+1}|-|\Gamma_{\star,h}^m|\bigr)
\le-\tau\beta\|\nabla_{\Gamma_h^m}(\gamma_0H_h^{m+1})\|_{L^2(\Gamma_h^m)}^2,
\]
which is \eqref{eq:area-decay-1}.
\end{proof}

\begin{theorem}[Energy stability of the full-substrate scheme]
\label{thm:energy-unified}
The scheme \eqref{eq:scheme-unified} satisfies
\[
W_h^{m+1}+\tau\beta\|\nabla_{\Gamma_h^m}(\gamma_0H_h^{m+1})\|_{L^2(\Gamma_h^m)}^2\le W_h^m,
\]
with $W_h^m$ as in \eqref{eq:full-energy}.
\end{theorem}

\begin{proof}
Similarly to the proof of Theorem~\ref{thm:energy-neg},
choosing \(\varphi_h=\gamma_0H_h^{m+1}\) in \eqref{eq:scheme-unified-flow}
and \(w_h=v_h^{m+1}\) in \eqref{eq:scheme-unified-curv} yields
\[
\begin{aligned}
&\gamma_0\int_{\Gamma_h^m}\nabla_{\Gamma_h^m}X_h^{m+1}\cdot
\nabla_{\Gamma_h^m}v_h^{m+1}
+\gamma_{\rm in}\int_{\Gamma_{{\rm in},h}^m}
\nabla_{\Gamma_{{\rm in},h}^m}X_h^{m+1}\cdot
\nabla_{\Gamma_{{\rm in},h}^m}v_h^{m+1}\\
&\quad
+\gamma_{\rm out}\int_{\Gamma_{{\rm out},h}^m}
\nabla_{\Gamma_{{\rm out},h}^m}X_h^{m+1}\cdot
\nabla_{\Gamma_{{\rm out},h}^m}v_h^{m+1}
=-\beta\|\nabla_{\Gamma_h^m}(\gamma_0H_h^{m+1})\|_{L^2(\Gamma_h^m)}^2.
\end{aligned}
\]
All three stiffness terms carry positive coefficients, so multiplying by
\(\tau\) and applying the area inequality \eqref{eq:area-ineq-neg} on
\(\Gamma_h^m\), \(\Gamma_{{\rm in},h}^m\), and \(\Gamma_{{\rm out},h}^m\) bounds \(W_h^{m+1}-W_h^m\) by \(-\tau\beta\|\nabla_{\Gamma_h^m}(\gamma_0H_h^{m+1})\|_{L^2(\Gamma_h^m)}^2\), which is
the stated inequality.
\end{proof}

\subsection{The Lagrange-multiplier formulation} \label{subsec:multiplier}

So far, the attachment constraint has been built
into the velocity spaces \eqref{eq:subspace} and \eqref{eq:outer-subspace}.
We now impose it by a Lagrange multiplier.  The velocity space is then
a standard finite element space, and the substrate normal appears only in the
constraint and in the multiplier term.  This makes it easy to change the
substrate normal, which we use in Section~\ref{sec:iterative}.  We give the
details for the full-substrate scheme.  The single-patch schemes are treated in the same way, with
$\Gamma_{\star,h}^m$ in place of $\Gamma_{{\rm base},h}^m$ and
$W_{\star,h}^m$ in place of $W_h^m$.

A vertex \(q\in\mathcal Q_{{\rm base},h}^m\) is called active if
\(q\notin\partial\Gamma_{{\rm base},h}^m\).  Let
\begin{align}\label{subspace1}
\mathbb V_h^m:=\bigl\{\chi_h\in(S_h(\Gamma_{\cup{\rm base},h}^m))^3:\
\chi_h(q)=0\ \text{for all vertices }q\in\partial\Gamma_{{\rm base},h}^m\bigr\}
\end{align}
be the velocity space without the substrate attachment constraint, but with
the velocity pinned on the outer boundary, and let
\begin{align}\label{eq:Lambda-space}
\Lambda_h^m:=\bigl\{\psi_h\in S_h(\Gamma_{{\rm base},h}^m):\ \psi_h(q)=0\ \text{for all vertices }q\in\partial\Gamma_{{\rm base},h}^m\bigr\}
\end{align}
be the scalar multiplier space.  Let
\begin{align}\label{eq:Sh-defn}
\mathcal A_h(X^{m+1}_h,w_h)
:={}&\gamma_0\int_{\Gamma_h^m}\nabla_{\Gamma_h^m}X_h^{m+1}\cdot\nabla_{\Gamma_h^m}w_h
+\gamma_{\rm in}\int_{\Gamma_{{\rm in},h}^m}\nabla_{\Gamma_{{\rm in},h}^m}X_h^{m+1}\cdot\nabla_{\Gamma_{{\rm in},h}^m}w_h\notag\\
&+\gamma_{\rm out}\int_{\Gamma_{{\rm out},h}^m}\nabla_{\Gamma_{{\rm out},h}^m}X_h^{m+1}\cdot\nabla_{\Gamma_{{\rm out},h}^m}w_h
\end{align}
be the right-hand side of \eqref{eq:scheme-unified-curv}.
The mixed scheme reads: given $\Gamma_h^m$, $\Gamma_{{\rm in},h}^m$,
$\Gamma_{{\rm out},h}^m$, find
\[
(v_h^{m+1},H_h^{m+1},\lambda_h^{m+1})\in\mathbb V_h^m\times S_h(\Gamma_h^m)\times\Lambda_h^m
\]
such that, with \(X_h^{m+1}=\id+\tau v_h^{m+1}\) and for all
\((\varphi_h,w_h,\psi_h)\in
S_h(\Gamma_h^m)\times\mathbb V_h^m\times\Lambda_h^m\),
\begin{subequations}\label{eq:mixed}
\begin{align}
&\int_{\Gamma_h^m}^h(v_h^{m+1}\cdot n_h^m)\,\varphi_h
+\beta\int_{\Gamma_h^m}\nabla_{\Gamma_h^m}(\gamma_0H_h^{m+1})\cdot\nabla_{\Gamma_h^m}\varphi_h=0, \label{eq:mixed-a}\\[2pt]
&\int_{\Gamma_{{\rm base},h}^m}^h(v_h^{m+1}\cdot\bar n_{{\rm base},h}^m)\,\psi_h=0, \label{eq:mixed-c}\\[2pt]
&\gamma_0\int_{\Gamma_h^m}^hH_h^{m+1}\,n_h^m\cdot w_h
+\int_{\Gamma_{{\rm base},h}^m}^h\lambda_h^{m+1}\,\bar n_{{\rm base},h}^m\cdot w_h
=\mathcal A_h(X^{m+1}_h,w_h). \label{eq:mixed-b}
\end{align}
\end{subequations}
Equation \eqref{eq:mixed-c} imposes the attachment constraint at the active
vertices, and the multiplier term in \eqref{eq:mixed-b} is the matching normal
reaction of the substrate.

\begin{proposition}[Equivalence between the schemes~\eqref{eq:scheme-unified} and~\eqref{eq:mixed}]
\label{prop:equivalence}
Assume the hypotheses of Theorem~\ref{thm:wellposed-unified}. Then
\((v_h^{m+1},H_h^{m+1},\lambda_h^{m+1})\) solves \eqref{eq:mixed} for some
\(\lambda_h^{m+1}\in\Lambda_h^m\) if and only if \((v_h^{m+1},H_h^{m+1})\)
solves \eqref{eq:scheme-unified}, and \(\lambda_h^{m+1}\) is then unique.
Consequently, the mixed scheme has a unique solution.
\end{proposition}

\begin{proof}
{Clearly a solution of \eqref{eq:mixed} solves
\eqref{eq:scheme-unified}: by mass lumping, \eqref{eq:mixed-c} means
\(v_h^{m+1}\in\mathring{\mathcal V}_{\cup{\rm base},h}^m\), on which the
multiplier term in \eqref{eq:mixed-b} vanishes.  Conversely,
\eqref{eq:scheme-unified-curv} says that the residual
\(\mathcal A_h(X^{m+1}_h,w_h)-\gamma_0\int_{\Gamma_h^m}^hH_h^{m+1}n_h^m\cdot w_h\)
vanishes on \(\mathring{\mathcal V}_{\cup{\rm base},h}^m\), the common kernel of
the linearly independent functionals
\(w_h\mapsto w_h(q)\cdot\bar n_{{\rm base},h}^m(q)\); it is therefore a unique
combination of them, which defines \(\lambda_h^{m+1}\) at the active vertices.
Theorem~\ref{thm:wellposed-unified} gives the last claim that the mixed scheme has a unique solution.}
\end{proof}

\section{A nonlinear scheme for nodal substrate attachment}
\label{sec:iterative}

The linear schemes of Section~\ref{sec:schemes} approximate
the attachment condition through the fixed nodal normal $\bar n_{{\rm base},h}^m$, so the
updated substrate vertices need not stay on $\Gamma_{\rm base}$. As an alternative,
we now choose a substrate normal that depends on the unknown update instead.
This gives a nonlinear scheme that guarantees that each vertex of the substrate
mesh remains attached to $\Gamma_{\rm base}$. The resulting nonlinear
system can be solved by a Picard-type fixed-point iteration, where each iterate
can be shown to be energy stable.
For sake of brevity, we only give the details for the
full-substrate scheme, concentrating on the Lagrange-multiplier form \eqref{eq:mixed}.

\subsection{Energy stability is independent of the discrete normals}
\label{subsec:universal}

A key observation is that the energy estimates of Section~\ref{sec:schemes} do not use the values of $n_h^m$ and $\bar n_{{\rm base},h}^m$.
For the sake of stability, both normals may therefore be replaced by any nodal fields.

Let $N_h$ be a nodal field on $\Gamma_h^m$ and $M_h$ a nodal field on
$\Gamma_{{\rm base},h}^m$, and replace $n_h^m$ by $N_h$ and
$\bar n_{{\rm base},h}^m$ by $M_h$ in \eqref{eq:mixed}. {{The mass-lumped
products use only nodal values, so the original scheme is the case}
$N_h(q)=\hat n_h^m(q)$, $M_h(q)=\bar n_{{\rm base},h}^m(q)$.} The generalized
mixed scheme reads: find
$(v_h^{m+1},H_h^{m+1},\lambda_h^{m+1})\in
\mathbb V_h^m\times S_h(\Gamma_h^m)\times\Lambda_h^m$ such that, for all
$(\varphi_h,w_h,\psi_h)\in S_h(\Gamma_h^m)\times\mathbb V_h^m\times\Lambda_h^m$,
\begin{subequations}\label{eq:mixed-gen}
\begin{align}
&\int_{\Gamma_h^m}^h(v_h^{m+1}\cdot N_h)\,\varphi_h
+\beta\int_{\Gamma_h^m}\nabla_{\Gamma_h^m}(\gamma_0H_h^{m+1})\cdot\nabla_{\Gamma_h^m}\varphi_h=0,
\label{eq:mixed-gen-a}\\
&\int_{\Gamma_{{\rm base},h}^m}^h(v_h^{m+1}\cdot M_h)\,\psi_h=0,
\label{eq:mixed-gen-c}\\
&\gamma_0\int_{\Gamma_h^m}^hH_h^{m+1}\,N_h\cdot w_h
+\int_{\Gamma_{{\rm base},h}^m}^h\lambda_h^{m+1}\,M_h\cdot w_h
=\mathcal A_h(X^{m+1}_h,w_h),
\label{eq:mixed-gen-b}
\end{align}
\end{subequations}
{where $\mathcal A_h(X^{m+1}_h,w_h)$, defined in \eqref{eq:Sh-defn}, is the right-hand
side of \eqref{eq:mixed-b}.}

\begin{lemma}[Energy stability for general nodal normals]
\label{lem:universal}
Assume the geometric conditions of Theorem~\ref{thm:wellposed-unified} and the following conditions on \(N_h\) and \(M_h\): $N_h(q)\neq0$ at every vertex of
$\Gamma_h^m$, and these values span $\mathbb R^3$; $M_h(q)\neq0$ at every active substrate vertex; and
$N_h(q),M_h(q)$ are linearly independent at every contact-line vertex. Then
\eqref{eq:mixed-gen} has a unique solution, and it satisfies
\[
W_h^{m+1}+\tau\beta\|\nabla_{\Gamma_h^m}(\gamma_0H_h^{m+1})\|_{L^2(\Gamma_h^m)}^2\le W_h^m
\]
with $W_h^m$ as in \eqref{eq:full-energy}.
\end{lemma}

\begin{proof}
Unique solvability follows as in
Theorem~\ref{thm:wellposed-unified} and Proposition~\ref{prop:equivalence},
with the hypotheses on \((N_h,M_h)\) in place of {\rm(A2)}--{\rm(A3)}{, and with \(M_h\) in place of \(\bar n_{{\rm base},h}^m\)}.
Take \(w_h=v_h^{m+1}\) in \eqref{eq:mixed-gen-b},
\(\varphi_h=\gamma_0H_h^{m+1}\) in \eqref{eq:mixed-gen-a}, and
\(\psi_h=\lambda_h^{m+1}\) in \eqref{eq:mixed-gen-c}. The constraint
\eqref{eq:mixed-gen-c} cancels the multiplier term
\(\int_{\Gamma_{{\rm base},h}^m}^h\lambda_h^{m+1}M_h\cdot v_h^{m+1}\),
and the flow equation \eqref{eq:mixed-gen-a} turns the curvature term into
\(-\beta\|\nabla_{\Gamma_h^m}(\gamma_0H_h^{m+1})\|_{L^2(\Gamma_h^m)}^2\). Then
\[
\mathcal A_h(X^{m+1}_h,v_h^{m+1})
=-\beta\|\nabla_{\Gamma_h^m}(\gamma_0H_h^{m+1})\|_{L^2(\Gamma_h^m)}^2\le0.
\]
All three stiffness terms in \(\mathcal A_h(X^{m+1}_h,v_h^{m+1})\) carry positive
coefficients, so multiplying by \(\tau\) and applying the area inequality \eqref{eq:area-ineq-neg} on
\(\Gamma_h^m\), \(\Gamma_{{\rm in},h}^m\), and \(\Gamma_{{\rm out},h}^m\)
gives the inequality.
\end{proof}

\begin{remark}
\label{rem:non-normalized}
{Let \(a\) be a nodal function with \(a(q)\ne0\) at every active vertex.
Because the lumped integrals only use nodal values, replacing \(M_h\) by
\(aM_h\) and \(\lambda_h^{m+1}\) by \(\lambda_h^{m+1}/a\) leaves
\eqref{eq:mixed-gen-c} and the term \(\lambda_h^{m+1}M_h\) in
\eqref{eq:mixed-gen-b} unchanged, so \(v_h^{m+1}\) and \(H_h^{m+1}\) do not
change.  In particular, the normalization in \eqref{eq:level-set-normal} is a
convention: \(M_h=\nabla\phi\) gives the same solution.}
\end{remark}

{Within the hypotheses of Lemma~\ref{lem:universal}, the choice of
$(N_h,M_h)$ therefore affects accuracy but not stability. We keep
$N_h=\hat n_h^m$ and in the following construct $M_h$ {so that the substrate vertices stay on the substrate.}}

\subsection{Nodal substrate attachment}
\label{subsec:averaged}

Let \(q\) be a substrate vertex.  We call
\(\mathcal M_q:\mathbb R^3\to\mathbb R^3\) an attachment map if
\begin{align}
\label{eq:discrete-chain-rule}
\tau\,\mathcal M_q(z)\cdot z
=
\phi(q+\tau z)-\phi(q).
\end{align}
With \(M_h(q):=\mathcal M_q(v_h^{m+1}(q))\), the attachment constraint
\(v_h^{m+1}(q)\cdot M_h(q)=0\) at an active vertex and
\eqref{eq:discrete-chain-rule} give
\(\phi(X_h^{m+1}(q))=\phi(q+\tau v_h^{m+1}(q))=\phi(q)\).  Two attachment maps
are the averaged secant and, for \(z\ne0\), the difference quotient
\begin{align}
\mathcal M_q^{\rm sec}(z)
&:=\int_0^1 \nabla\phi(q+s\tau z)\,\diff s ,
\label{eq:secant-normal}\\
\mathcal M_q^{\rm dq}(z)
&:=\nabla\phi(q)
+\frac{\phi(q+\tau z)-\phi(q)-\tau\nabla\phi(q)\cdot z}{\tau |z|^2}\,z ,
\label{eq:rank-one}
\end{align}
with \(\mathcal M_q^{\rm dq}(0):=\nabla\phi(q)\).  The first is computed exactly
by the midpoint rule when \(\phi\) is quadratic; the second needs no
quadrature.  Both satisfy \eqref{eq:discrete-chain-rule} and
\(|\mathcal M_q(z)-\nabla\phi(q)|\le\frac12\tau|z|\sup_{0\le s\le1}\|D^2\phi(q+s\tau z)\|\),
so for small \(\tau|z|\) the map \(\mathcal M_q(z)\) is nonzero and, at a
contact-line vertex, not parallel to \(N_h(q)\) by {\rm(A3)}.

\begin{theorem}[{Nodal attachment} and energy stability]
\label{theorem:exact}
Let \(N_h=\hat n_h^m\) and \(M_h(q)=\mathcal M_q(v_h^{m+1}(q))\) with
\(\mathcal M_q\) from \eqref{eq:secant-normal} or \eqref{eq:rank-one}, and let
\((v_h^{m+1},H_h^{m+1},\lambda_h^{m+1})\) solve \eqref{eq:mixed-gen}.  Then
\[
W_h^{m+1}+\tau\beta\|\nabla_{\Gamma_h^m}(\gamma_0H_h^{m+1})\|_{L^2(\Gamma_h^m)}^2\le W_h^m,
\qquad
\phi\bigl(X_h^{m+1}(q)\bigr)=\phi(q)
\]
at every substrate vertex \(q\).
\end{theorem}

\begin{proof}
{{The energy inequality follows as in the proof of
Lemma~\ref{lem:universal}, which uses no property of \(N_h\) and \(M_h\). At
active vertices the attachment identity is the computation after
\eqref{eq:discrete-chain-rule}. At pinned vertices it holds because
\(v_h^{m+1}(q)=0\).}}
\end{proof}

\subsection{The iterative scheme}
\label{subsec:iterative}

Because the attachment field depends on the unknown velocity, the scheme
in Theorem~\ref{theorem:exact} is nonlinear. {We solve it by the
Picard iteration of Algorithm~\ref{alg:picard}, which freezes the field at the
previous iterate, \(M_h^{(k)}(q)=\mathcal M_q(v_h^{(k)}(q))\) with
\(\mathcal M_q\) from \eqref{eq:secant-normal} or \eqref{eq:rank-one}. Each
step then solves the linear system \eqref{eq:mixed-gen} with
\(M_h=M_h^{(k)}\).  All iterates are computed on the meshes at
\(t_m\), so \(N_h=\hat n_h^m\) stays fixed and only \(M_h\) changes; the
meshes are moved once, after the last iterate.} The initial
guess \(v_h^{(0)}\) is obtained from the linear scheme \eqref{eq:mixed} with
the level-set normal \eqref{eq:level-set-normal} at the vertices of \(\Gamma_{{\rm base},h}^m\).

\makeatletter
{\renewcommand\fs@ruled{\def\@fs@cfont{\bfseries}%
  \let\@fs@capt\floatc@ruled
  \def\@fs@pre{{\color{black}\hrule height.8pt depth0pt}\kern2pt}%
  \def\@fs@post{\kern2pt{\color{black}\hrule}\relax}%
  \def\@fs@mid{\kern2pt{\color{black}\hrule}\kern2pt}%
  \let\@fs@iftopcapt\iftrue}
\begin{algorithm}[t]
\caption{A time step \(t_m\to t_{m+1}\).}
\label{alg:picard}
\begin{algorithmic}[1]
\Require
\(\Gamma_h^m\), \(\Gamma_{{\rm base},h}^m\), \(\phi\), \(\tau\),
and attachment tolerance \({\rm tol}_{\phi}\)
\State
solve \eqref{eq:mixed} with the level-set normal \eqref{eq:level-set-normal} for \(v_h^{(0)}\);
set \(k\gets0\)
\Repeat
  \State
  assemble \(M_h^{(k)}(q)=\mathcal M_q(v_h^{(k)}(q))\) using either
  \eqref{eq:secant-normal} or \eqref{eq:rank-one}
  \State
  solve \eqref{eq:mixed-gen} with \(N_h=\hat n_h^m\), \(M_h=M_h^{(k)}\) for
  \((v_h^{(k+1)},H_h^{(k+1)},\lambda_h^{(k+1)})\)
  \State \(k\gets k+1\)
\Until{
\(\displaystyle
\max_q
\left|
\phi(q+\tau v_h^{(k)}(q))-\phi(q)
\right|
\le {\rm tol}_{\phi}\)}
\State
set \(v_h^{m+1}\gets v_h^{(k)}\) and
\(X_h^{m+1}\gets\id+\tau v_h^{m+1}\)
\State
update all meshed surfaces with \(X_h^{m+1}\)
\end{algorithmic}
\end{algorithm}
}
\makeatother

\begin{theorem}[Energy stability and attachment at a fixed point]
\label{thm:iterative}
{Assume that \(N_h\) and every frozen field \(M_h^{(k)}\) satisfy the
conditions of Lemma~\ref{lem:universal}, so that every iterate is well
defined.}
Then:
\begin{enumerate}[label={\rm(\roman*)}]
\item
{Every Picard iterate is energy stable, \(W_h^{(k+1)}+\tau\beta\|\nabla_{\Gamma_h^m}(\gamma_0H_h^{(k+1)})\|_{L^2(\Gamma_h^m)}^2\le W_h^m\),}
where \(W_h^{(k+1)}\) is the energy \eqref{eq:full-energy} after a
trial move by \(\id+\tau v_h^{(k+1)}\). The
estimate holds for every \(k\), regardless of the stopping tolerance.

\item
{A fixed point $v_h^{m+1}$ of the iteration satisfies
$M_h(q)=\mathcal M_q(v_h^{m+1}(q))$ and therefore solves the nonlinear system
of Theorem~\ref{theorem:exact}. In particular
$\phi(q+\tau v_h^{m+1}(q))=\phi(q)$ at every substrate vertex, so an initially
attached mesh remains exactly attached.}
\end{enumerate}
\end{theorem}

\begin{proof}
{For each \(k\) the frozen system is an instance of \eqref{eq:mixed-gen},
so part~(i) is Lemma~\ref{lem:universal}; part~(ii) is
Theorem~\ref{theorem:exact} applied at the fixed point.}
\end{proof}

\section{Fully anisotropic surface energies}
\label{sec:anisotropic}

We now let {the film density \(\gamma\) and the substrate densities
\(\gamma_{\rm in}\) and \(\gamma_{\rm out}\)} depend on the orientation, using the class of
anisotropies of \cite{BGN2008plane,BGN2008aniso}, and prove that the resulting scheme is
unconditionally energy stable.

\subsection{Anisotropic energies}
\label{subsec:aniso-energy}

Following \cite{BGN2008plane,BGN2008aniso}, let \(\underline G_\ell\in\mathbb R^{3\times3}\),
\(\ell=1,\dots,L\), be symmetric positive definite, and let the film
density and its Cahn--Hoffman vector \cite{CahnHoffman1974} be
\begin{align}\label{eq:aniso-gamma}
\gamma(\vec p):=\sum_{\ell=1}^L\gamma_\ell(\vec p),\qquad
\gamma_\ell(\vec p):=\sqrt{\vec p\cdot\underline G_\ell\,\vec p},
\end{align}
\begin{align}\label{eq:cahn-hoffman}
\vec\xi(\vec p):=\gamma'(\vec p)=\sum_{\ell=1}^L\frac{\underline G_\ell\vec p}{\gamma_\ell(\vec p)},
\qquad \vec\xi(\vec p)\cdot\vec p=\gamma(\vec p).
\end{align}
Both are defined for \(\vec p\ne0\), and \(\gamma\) is one-homogeneous; the
isotropic case \(\gamma(\vec p)=\gamma_0|\vec p|\) is that of
Sections~\ref{sec:continuous}--\ref{sec:iterative}.  The substrate densities
\(\gamma_{\rm in}\) and \(\gamma_{\rm out}\) have the same form with their own
parameters, are evaluated at \(n_b\), and need not agree on the contact line.
For \(\gamma_\star\in\{\gamma,\gamma_{\rm in},\gamma_{\rm out}\}\) we write
\(L_{\gamma_\star}\), \(\underline G_{\gamma_\star,\ell}\),
\(\widetilde{\underline G}_{\gamma_\star,\ell}\), \(\vec\xi_{\gamma_\star}\)
for its quantities.  The total energy is
\begin{align}\label{eq:aniso-energy-W}
W^\gamma(t):=\int_{\Gamma(t)}\gamma(n)+\int_{\Gamma_{\rm in}(t)}\gamma_{\rm in}(n_b)+\int_{\Gamma_{\rm out}(t)}\gamma_{\rm out}(n_b).
\end{align}
With \(\widetilde{\underline G}_\ell:=[\det\underline G_\ell]^{1/2}\underline G_\ell^{-1}\)
and \((\vec u,\vec v)_{\widetilde{\underline G}_\ell}:=\vec u\cdot\widetilde{\underline G}_\ell\vec v\),
let \(\{t_1,t_2\}\) be a tangent basis orthonormal in
\((\cdot,\cdot)_{\widetilde{\underline G}_\ell}\).  The anisotropic surface
gradient, divergence, and pairing are
\(\nabla_s^{\widetilde G_\ell}f:=\sum_i(\partial_{t_i}f)\,t_i\),
\(\nabla_s^{\widetilde G_\ell}\cdot\vec f:=\sum_i(\partial_{t_i}\vec f,t_i)_{\widetilde{\underline G}_\ell}\), and
\((\nabla_s^{\widetilde G_\ell}\vec\chi,\nabla_s^{\widetilde G_\ell}w)_{\widetilde{\underline G}_\ell}:=\sum_i(\partial_{t_i}\vec\chi,\partial_{t_i}w)_{\widetilde{\underline G}_\ell}\)
\cite{BGN2008aniso,BGN2020handbook}.

\subsection{Anisotropic curvature and the Young--Herring relation}
\label{subsec:aniso-curvature}

For a \(C^2\) surface \(\Sigma\) with unit normal \(n_{\scriptscriptstyle\Sigma}\) and
density \(\gamma_\star\), the anisotropic mean curvature
\(\varkappa_{\gamma_\star}:=\nabla_\Sigma\cdot\vec\xi_{\gamma_\star}(n_{\scriptscriptstyle\Sigma})\)
satisfies \cite[Theorem~2.1]{BGN2008aniso}
\begin{align}\label{eq:aniso-id-H}
-\varkappa_{\gamma_\star}\,n_{\scriptscriptstyle\Sigma}
=\sum_{\ell=1}^{L_{\gamma_\star}}\gamma_{\star,\ell}(n_{\scriptscriptstyle\Sigma})\,\widetilde{\underline G}_{\gamma_\star,\ell}\,\nabla_s^{\widetilde G_{\gamma_\star,\ell}}\!\cdot\!\nabla_s^{\widetilde G_{\gamma_\star,\ell}}\,\id\quad\text{on }\Sigma,
\end{align}
which for \(\gamma_\star(\vec p)=|\vec p|\) reduces to \(-Hn_{\scriptscriptstyle\Sigma}=\Delta_\Sigma\id\).
The first variation of the individual summands of the energy $W^\gamma$ can be written in the weak form \cite{BGN2008aniso,BGN2010clusters}:
\begin{align}\label{eq:aniso-area-var}
\delta\Bigl(\int_\Sigma\gamma_\star(n_{\scriptscriptstyle\Sigma})\Bigr)(w)=\bigl\langle\nabla_s^{\widetilde G_{\gamma_\star}}\id,\nabla_s^{\widetilde G_{\gamma_\star}}w\bigr\rangle_{\Sigma,\gamma_\star},
\end{align}
where
\begin{align}\label{eq:bilinear}
\bigl\langle\nabla_s^{\widetilde G_{\gamma_\star}}\vec\chi,\nabla_s^{\widetilde G_{\gamma_\star}}w\bigr\rangle_{\Sigma,\gamma_\star}
:=\sum_{\ell=1}^{L_{\gamma_\star}}\int_\Sigma\bigl(\nabla_s^{\widetilde G_{\gamma_\star,\ell}}\vec\chi,\nabla_s^{\widetilde G_{\gamma_\star,\ell}}w\bigr)_{\widetilde{\underline G}_{\gamma_\star,\ell}}\gamma_{\star,\ell}(n_{\scriptscriptstyle\Sigma}).
\end{align}
If
\(\mu_{\scriptscriptstyle\Sigma,\scriptscriptstyle\partial}\) is the outward
conormal and \(\tau_{\scriptscriptstyle\Sigma}\) the unit tangent of
\(\partial\Sigma\) with
\(n_{\scriptscriptstyle\Sigma}=\mu_{\scriptscriptstyle\Sigma,\scriptscriptstyle\partial}\times\tau_{\scriptscriptstyle\Sigma}\),
the \(\widetilde{\underline G}_{\gamma_\star,\ell}\)-divergence theorem for the right hand side of \eqref{eq:bilinear}, see \cite{BGN2008aniso, BGN2010clusters}, gives a
boundary term involving the anisotropic conormals
\begin{align}\label{eq:mu-beta}
\mu_{\gamma_\star}:=\gamma_\star(n_{\scriptscriptstyle\Sigma})\,\mu_{\scriptscriptstyle\Sigma,\scriptscriptstyle\partial}-\bigl(\vec\xi_{\gamma_\star}(n_{\scriptscriptstyle\Sigma})\cdot\mu_{\scriptscriptstyle\Sigma,\scriptscriptstyle\partial}\bigr)\,n_{\scriptscriptstyle\Sigma}
=\tau_{\scriptscriptstyle\Sigma}\times\vec\xi_{\gamma_\star}(n_{\scriptscriptstyle\Sigma}).
\end{align}
With the conormals \eqref{eq:mu-beta} of \(\Gamma\), \(\Gamma_{\rm in}\),
\(\Gamma_{\rm out}\) along \(\partial\Gamma\), the contact-line condition is the
anisotropic Young--Herring relation
\begin{align}\label{eq:aniso-young}
\bigl(\mu_\gamma+\mu_{\gamma_{\rm in}}+\mu_{\gamma_{\rm out}}\bigr)\cdot\mu_b=0
\qquad\text{on }\partial\Gamma,
\end{align}
equivalently
\(P_{T\Gamma_{\rm base}}(\mu_\gamma+\mu_{\gamma_{\rm in}}+\mu_{\gamma_{\rm out}})=0\),
since all three conormals are orthogonal to the contact line; in the isotropic
case it is \eqref{eq:young-energy}.  Summing \eqref{eq:aniso-area-var} over
\(\Gamma\), \(\Gamma_{\rm in}\), \(\Gamma_{\rm out}\) gives \(\delta W^\gamma(w)\),
and the weak formulation requires, for all
\(w\in\mathcal V_{\Gamma\cup\Gamma_{\rm base}}\) of \eqref{eq:space-W},
\begin{multline}\label{eq:aniso-W}
\int_\Gamma\varkappa_\gamma\,n\cdot w=\bigl\langle\nabla_s^{\widetilde G}\id,\nabla_s^{\widetilde G}w\bigr\rangle_{\Gamma,\gamma}\\+\bigl\langle\nabla_s^{\widetilde G_{\rm in}}\id,\nabla_s^{\widetilde G_{\rm in}}w\bigr\rangle_{\Gamma_{\rm in},\gamma_{\rm in}}+\bigl\langle\nabla_s^{\widetilde G_{\rm out}}\id,\nabla_s^{\widetilde G_{\rm out}}w\bigr\rangle_{\Gamma_{\rm out},\gamma_{\rm out}}.
\end{multline}

The free surface evolves by anisotropic surface diffusion
\begin{align}\label{eq:aniso-sd}
v\cdot n=\nabla_\Gamma\cdot\bigl(\beta(n)\nabla_\Gamma\varkappa_\gamma\bigr)\quad\text{on }\Gamma,\qquad
\beta(n)\,\mu_{\scriptscriptstyle\partial}\cdot\nabla_\Gamma\varkappa_\gamma=0\quad\text{on }\partial\Gamma,
\end{align}
where the mobility $\beta>0$ may depend on the orientation \cite{TaylorCahn1994,BGN2020handbook,GarckeKnopfNurnbergZhao2023}. For $\gamma(\vec p)=\gamma_0|\vec p|$ and constant $\beta$ we have $\varkappa_\gamma=\gamma_0H$, and \eqref{eq:aniso-sd} is \eqref{eq:sd}.

\begin{remark}
\label{rem:aniso-no-SP}
With anisotropic substrate densities, a single-patch formulation as in \eqref{eq:W1-equiv} or
\eqref{eq:W2-equiv} is in general no longer possible, because the density difference
\(\gamma_{\rm in}-\gamma_{\rm out}\) depends on the orientation and is not constant.
However, if \(\gamma_{\rm in}\) and \(\gamma_{\rm out}\) are both constant, then the
single-patch schemes can be applied to an anisotropic film: their energy
\(\int_{\Gamma_h^m}\gamma(n_h^m)+\sigma|\Gamma_{\star,h}^m|\) decreases, by
Lemma~\ref{lem:aniso-area}, below, and \eqref{eq:area-ineq-neg}.
If, in addition, \(\gamma_{\rm in}=\gamma_{\rm out}\), it suffices to mesh the
free surface (Remark~\ref{rem:equal-densities} and
Figure~\ref{fig:aniso-sine}).
\end{remark}

\subsection{Anisotropic full-substrate scheme}
\label{subsec:aniso-scheme}

We use the multiplier form of \S\ref{subsec:multiplier} with
nodal fields $N_h$ on $\Gamma_h^m$ and $M_h$ on $\Gamma_{{\rm base},h}^m$, as
in \eqref{eq:mixed-gen}.
Let $n_h^m$,
$n_{{\rm in},h}^m$, and $n_{{\rm out},h}^m$ be the elementwise unit normals of
the meshes $\Gamma_h^m$, $\Gamma_{{\rm in},h}^m$, and $\Gamma_{{\rm out},h}^m$.
They are used only to evaluate the densities and the mobility: in the weights
$\gamma_{\star,\ell}(n_\Sigma)$ of \eqref{eq:bilinear}, with
$\gamma_\star=\gamma,\gamma_{\rm in},\gamma_{\rm out}$ on the three meshes, in $\beta(n_h^m)$ in \eqref{eq:aniso-scheme-a}, and in the
discrete energy \eqref{eq:aniso-W-discrete} below, as required by
Lemma~\ref{lem:aniso-area}.  We set
\begin{align}\label{eq:aniso-Sh}
\mathcal A_h^\gamma(X^{m+1}_h,w_h):={}&
\bigl\langle\nabla_s^{\widetilde G}X_h^{m+1},\nabla_s^{\widetilde G}w_h\bigr\rangle_{\Gamma_h^m,\gamma}
+\bigl\langle\nabla_s^{\widetilde G_{\rm in}}X_h^{m+1},\nabla_s^{\widetilde G_{\rm in}}w_h\bigr\rangle_{\Gamma_{{\rm in},h}^m,\gamma_{\rm in}}\notag\\
&+\bigl\langle\nabla_s^{\widetilde G_{\rm out}}X_h^{m+1},\nabla_s^{\widetilde G_{\rm out}}w_h\bigr\rangle_{\Gamma_{{\rm out},h}^m,\gamma_{\rm out}} .
\end{align}
Find \((v_h^{m+1},\varkappa_{\gamma,h}^{m+1},\lambda_h^{m+1})\in\mathbb V_h^m\times S_h(\Gamma_h^m)\times\Lambda_h^m\)
such that, with $X_h^{m+1}=\id+\tau v_h^{m+1}$ and for all
$(\varphi_h,w_h,\psi_h)\in S_h(\Gamma_h^m)\times\mathbb V_h^m\times\Lambda_h^m$,
\begin{subequations}\label{eq:aniso-scheme}
\begin{align}
&\int_{\Gamma_h^m}^h(v_h^{m+1}\cdot N_h)\,\varphi_h
+\int_{\Gamma_h^m}\beta(n_h^m)\,\nabla_{\Gamma_h^m}\varkappa_{\gamma,h}^{m+1}\cdot\nabla_{\Gamma_h^m}\varphi_h=0, \label{eq:aniso-scheme-a}\\
&\int_{\Gamma_{{\rm base},h}^m}^h(v_h^{m+1}\cdot M_h)\,\psi_h=0, \label{eq:aniso-scheme-c}\\
&\int_{\Gamma_h^m}^h\varkappa_{\gamma,h}^{m+1}\,N_h\cdot w_h
+\int_{\Gamma_{{\rm base},h}^m}^h\lambda_h^{m+1}\,M_h\cdot w_h
=\mathcal A_h^\gamma(X^{m+1}_h,w_h). \label{eq:aniso-scheme-b}
\end{align}
\end{subequations}
This is \eqref{eq:mixed-gen} with $\gamma_0H_h^{m+1}$ replaced by
$\varkappa_{\gamma,h}^{m+1}$, $\beta$ by $\beta(n_h^m)$, and $\mathcal A_h$ by $\mathcal A_h^\gamma$.
Unlike the mesh normals in $\mathcal A_h^\gamma$, the field $M_h$ only enters
the attachment condition and can be freely chosen.  As in
Section~\ref{sec:iterative}, we take $N_h=\hat n_h^m$ and the attachment field
$M_h=M_h^{(k)}$ of Algorithm~\ref{alg:picard}, with \eqref{eq:aniso-scheme} in
every solve.  The discrete energy is
\begin{align}\label{eq:aniso-W-discrete}
W^{\gamma,m}_h:=\int_{\Gamma_h^m}\gamma(n_h^m)+\int_{\Gamma_{{\rm in},h}^m}\gamma_{\rm in}(n_{{\rm in},h}^m)+\int_{\Gamma_{{\rm out},h}^m}\gamma_{\rm out}(n_{{\rm out},h}^m).
\end{align}

\subsection{Energy stability}
\label{subsec:aniso-energy-stable}

\begin{lemma}[{Area inequality \cite[Theorem~3.2]{BGN2008aniso}}]
\label{lem:aniso-area}
Let \(\Sigma_h\) be a polyhedral surface, let \(\gamma_\star\) be of the form
\eqref{eq:aniso-gamma}, and let
\(|\Sigma_h|_{\gamma_\star}:=\int_{\Sigma_h}\gamma_\star(n_{\Sigma_h})\), with
\(|\vec X(\Sigma_h)|_{\gamma_\star}\) defined with the normal of
\(\vec X(\Sigma_h)\).  Then for every \(\vec X\in(S_h(\Sigma_h))^3\),
\begin{align}\label{eq:aniso-area-ineq}
\bigl\langle\nabla_s^{\widetilde G_{\gamma_\star}}\vec X,\nabla_s^{\widetilde G_{\gamma_\star}}(\vec X-\id)\bigr\rangle_{\Sigma_h,\gamma_\star}\ge|\vec X(\Sigma_h)|_{\gamma_\star}-|\Sigma_h|_{\gamma_\star}.
\end{align}
\end{lemma}

\begin{theorem}[Anisotropic energy stability]
\label{thm:aniso-energy}
Let $N_h$ and $M_h$ satisfy the hypotheses of
Lemma~\ref{lem:universal}.  Then the scheme \eqref{eq:aniso-scheme} has a
unique solution, and it satisfies
\[
W^{\gamma,m+1}_h+\tau\|\beta(n_h^m)^{1/2}\nabla_{\Gamma_h^m}\varkappa_{\gamma,h}^{m+1}\|_{L^2(\Gamma_h^m)}^2\le W^{\gamma,m}_h.
\]
If \eqref{eq:aniso-scheme} is used in every solve of
Algorithm~\ref{alg:picard}, this holds for every Picard iterate, and a fixed
point keeps the substrate vertices exactly on the substrate.
\end{theorem}

\begin{proof}
Unique solvability follows as in Lemma~\ref{lem:universal}.
Its proof needs only that the three stiffness forms are nonnegative and vanish
exactly on the fields that are constant on each connected component, and the
forms \eqref{eq:bilinear} have both properties.
Testing \eqref{eq:aniso-scheme-a} with $\varphi_h=\varkappa_{\gamma,h}^{m+1}$,
\eqref{eq:aniso-scheme-c} with $\psi_h=\lambda_h^{m+1}$, and
\eqref{eq:aniso-scheme-b} with $w_h=v_h^{m+1}$ gives, over
the surfaces $\Sigma\in\{\Gamma_h^m,\Gamma_{{\rm in},h}^m,\Gamma_{{\rm out},h}^m\}$
{with the densities}
$\gamma_\star\in\{\gamma,\gamma_{\rm in},\gamma_{\rm out}\}$,
\[
\sum_{\Sigma}\tfrac1\tau\bigl\langle\nabla_s^{\widetilde G_{\gamma_\star}}X_h^{m+1},
\nabla_s^{\widetilde G_{\gamma_\star}}(X_h^{m+1}-\id)\bigr\rangle_{\Sigma,\gamma_\star}
=-\int_{\Gamma_h^m}\beta(n_h^m)\,|\nabla_{\Gamma_h^m}\varkappa_{\gamma,h}^{m+1}|^2 ,
\]
and applying Lemma~\ref{lem:aniso-area} on each $\Sigma$ then yields
the stated inequality.  The last statement
follows as in Theorem~\ref{thm:iterative}.
\end{proof}

{
\section{Extension to other tangential motions: dual MDR}
\label{sec:dual-mdr}
In the schemes of Sections~\ref{sec:schemes}--\ref{sec:anisotropic} the
vertices move along the free surface by the original BGN tangential motion, and along the substrate patches by the level-set BGN motion of
\S\ref{subsec:single-patch}.  To see what determines this motion,
let
\begin{equation}\label{eq:Th-space}
\begin{aligned}
\mathcal T_h^m:=\bigl\{\eta_h\in\mathbb V_h^m:\ &\eta_h(q)\cdot N_h(q)=0\
\text{at every vertex $q$ of }\Gamma_h^m,\\
&\eta_h(q)\cdot M_h(q)=0\ \text{at every active substrate vertex $q$}\bigr\}
\end{aligned}
\end{equation}
be the velocities that are tangent to the surfaces at the vertices.  Testing
\eqref{eq:mixed-gen-b} with $\eta_h\in\mathcal T_h^m$, and noting that $X^{m+1}_h = \id+\tau v_h^{m+1}$, gives
\begin{equation}\label{eq:lsbgn-tangential}
\mathcal A_h(v_h^{m+1},\eta_h)=-\tau^{-1}\mathcal A_h(\id,\eta_h)
\qquad\forall\eta_h\in\mathcal T_h^m .
\end{equation}
Since $\mathcal A_h(\id,\eta_h)\neq0$ in general, the tangential velocity is of
order $1/\tau$, and for small $\tau$ the substrate vertices may cluster on a
curved substrate in three dimensions (\S\ref{subsec:num-dual}).
Following \cite{GaoLiTang2026dual,HuLi2022MDR,GaoGarckeLiTang2025bgnmdr}, we
replace \eqref{eq:lsbgn-tangential} by the MDR tangential motion
\begin{align}\label{eq:dual-tangential}
\mathcal D_h(v_h^{m+1},\eta_h)=0\quad\forall\eta_h\in\mathcal T_h^m ,\qquad
\mathcal D_h(u_h,w_h):=\sum_{i=0}^{2}\int_{\Sigma_i^m}\nabla_{\Sigma_i^m}u_h\cdot
\nabla_{\Sigma_i^m}w_h ,
\end{align}
with $(\Sigma_0^m,\Sigma_1^m,\Sigma_2^m):=(\Gamma_h^m,\Gamma_{{\rm in},h}^m,
\Gamma_{{\rm out},h}^m)$.  Here $\mathcal D_h$ is the sum of the
deformation-rate forms of the three surfaces, each with weight $1$.  Unlike
\eqref{eq:lsbgn-tangential}, condition \eqref{eq:dual-tangential} has no
$1/\tau$ term.
\subsection{Continuous formulation}
\label{subsec:dual-scheme}

Let
$(\Sigma_0,\Sigma_1,\Sigma_2):=(\Gamma,\Gamma_{\rm in},\Gamma_{\rm out})$,
$(\gamma_1,\gamma_2):=(\gamma_{\rm in},\gamma_{\rm out})$, where
$\gamma_0,\gamma_1,\gamma_2>0$ are constants, and
\begin{align}\label{dual-tang-cont}
\begin{aligned}
&\mathcal A(u,w):=\sum_{i=0}^{2}\gamma_i\int_{\Sigma_i}\nabla_{\Sigma_i}u\cdot
\nabla_{\Sigma_i}w,\qquad
\mathcal D(u,w):=\sum_{i=0}^{2}\int_{\Sigma_i}\nabla_{\Sigma_i}u\cdot
\nabla_{\Sigma_i}w,\\
&\mathcal T:=\{\eta\in\mathcal V_{\Gamma\cup\Gamma_{\rm base}}:\ \eta\cdot n=0\
\text{on }\Gamma\},
\end{aligned}
\end{align}
with $\mathcal V_{\Gamma\cup\Gamma_{\rm base}}$ from \eqref{eq:space-W}.
Fields in $\mathcal T$ are tangent to the contact line on $\partial\Gamma$, provided $\Gamma$ meets $\Gamma_{\rm base}$ transversally, that is, $n$ and $n_b$ are linearly independent on $\partial\Gamma$.
We seek $(v,H,z)\in\mathcal V_{\Gamma\cup\Gamma_{\rm base}}\times H^1(\Gamma)
\times\mathcal T$ such that, for all
$(w,\varphi,\eta)\in\mathcal V_{\Gamma\cup\Gamma_{\rm base}}\times H^1(\Gamma)
\times\mathcal T$,
\begin{subequations}\label{eq:dual-cont}
\begin{align}
&\gamma_0\int_\Gamma Hn\cdot w=\mathcal A(\id,w)-\mathcal D(z,w),
\label{eq:dual-cont-force}\\
&\int_\Gamma(v\cdot n)\,\varphi+\beta\int_\Gamma\nabla_\Gamma(\gamma_0H)\cdot\nabla_\Gamma\varphi=0,
\label{eq:dual-cont-flow}\\
&\mathcal D(v,\eta)=0.
\label{eq:dual-cont-mdr}
\end{align}
\end{subequations}
Here \eqref{eq:dual-cont-flow} is the weak form of \eqref{eq:sd} with the
no-flux condition \eqref{eq:bc-flux}, and the attachment condition
\eqref{eq:bc-attach} is built into $\mathcal V_{\Gamma\cup\Gamma_{\rm base}}$.
As in \cite{GaoLiTang2026dual}, \eqref{eq:dual-cont-mdr} fixes the
tangential velocity: it is the optimality condition of the constrained
minimization of the isotropic deformation-rate energy
\[
\min\Bigl\{\tfrac12\mathcal D(u,u):\
u\in\mathcal V_{\Gamma\cup\Gamma_{\rm base}},\ u\cdot n=v\cdot n\ \text{on }\Gamma
\Bigr\},
\]
in which $v\cdot n$ is given by \eqref{eq:dual-cont-flow} and the substrate constraint is
built into $\mathcal V_{\Gamma\cup\Gamma_{\rm base}}$.  For
$w\in\mathcal T$ the left-hand side of \eqref{eq:dual-cont-force} vanishes.  Integration by parts on
each $\Sigma_i$ leaves only normal fields inside the patches and the conormals
$\mu_{\scriptscriptstyle\partial}$, $\mu_b$ on $\partial\Gamma$.  All are
orthogonal to $w \in \mathcal T$, so $\mathcal A(\id,w)=0$.  Hence
$\mathcal D(z,w)=0$ for all $w\in\mathcal T$, and $z=0$ since $\mathcal D$
is positive definite on $\mathcal T$.  So \eqref{eq:dual-cont-force}
reduces to \eqref{eq:var-W}, which makes $H$ the mean curvature and gives
Young's law.  The argument only uses that $\mathcal D$ is symmetric and
positive definite on $\mathcal T$, so any such form, for instance the
$\gamma$-weighted form $\mathcal A$, may replace it; only the tangential
velocity changes.

\subsection{The discrete dual-MDR full-substrate scheme}
We use the isotropic form $\mathcal D_h$ of \eqref{eq:dual-tangential}, the
discrete counterpart of $\mathcal D$ in \eqref{dual-tang-cont}.  It is
positive definite on $\mathcal T_h^m$ if the values of $N_h$ at the vertices of
$\Gamma_h^m$ span $\mathbb R^3$: if $\mathcal D_h(u_h,u_h)=0$ for
$u_h\in\mathcal T_h^m$, then $u_h$ is a constant $c$ on the connected glued union
$\Gamma_h^m\cup\Gamma_{{\rm base},h}^m$, and $c\cdot N_h(q)=0$ at every vertex
$q$ of $\Gamma_h^m$, so $c=0$.
Given $\Gamma_h^m$ and the
fields $N_h$, $M_h$, find
\[
(v_h^{m+1},H_h^{m+1},\lambda_h^{m+1},z_h^{m+1},\kappa_h^{m+1},\rho_h^{m+1})\in
\mathbb V_h^m\times S_h(\Gamma_h^m)\times\Lambda_h^m\times
\mathbb V_h^m\times S_h(\Gamma_h^m)\times\Lambda_h^m ,
\]
with $X_h^{m+1}=\id+\tau v_h^{m+1}$, such that, for all
\[
(w_h,\varphi_h,\psi_h,\eta_h,\chi_h,\zeta_h)\in
\mathbb V_h^m\times S_h(\Gamma_h^m)\times\Lambda_h^m\times
\mathbb V_h^m\times S_h(\Gamma_h^m)\times\Lambda_h^m ,
\]
\begin{subequations}\label{eq:dual}
\begin{align}
&\gamma_0\int_{\Gamma_h^m}^hH_h^{m+1}\,N_h\cdot w_h
+\int_{\Gamma_{{\rm base},h}^m}^h\lambda_h^{m+1}\,M_h\cdot w_h
\notag\\
&\qquad=\mathcal A_h(X_h^{m+1},w_h)-\mathcal D_h(z_h^{m+1},w_h),
\label{eq:dual-force}\\
&\int_{\Gamma_h^m}^h(v_h^{m+1}\cdot N_h)\,\varphi_h
+\beta\int_{\Gamma_h^m}\nabla_{\Gamma_h^m}(\gamma_0H_h^{m+1})\cdot\nabla_{\Gamma_h^m}\varphi_h=0,
\label{eq:dual-flow}\\
&\int_{\Gamma_{{\rm base},h}^m}^h(v_h^{m+1}\cdot M_h)\,\psi_h=0,
\label{eq:dual-constr}\\
&\mathcal D_h(v_h^{m+1},\eta_h)
=\int_{\Gamma_h^m}^h\kappa_h^{m+1}\,N_h\cdot\eta_h
+\int_{\Gamma_{{\rm base},h}^m}^h\rho_h^{m+1}\,M_h\cdot\eta_h,
\label{eq:dual-mdr}\\
&\int_{\Gamma_h^m}^h(z_h^{m+1}\cdot N_h)\,\chi_h=0,\qquad
\int_{\Gamma_{{\rm base},h}^m}^h(z_h^{m+1}\cdot M_h)\,\zeta_h=0.
\label{eq:dual-z}
\end{align}
\end{subequations}
Scheme \eqref{eq:dual} is the discrete counterpart of \eqref{eq:dual-cont},
written in multiplier form as in \S\ref{subsec:multiplier}:
\eqref{eq:dual-force}, \eqref{eq:dual-flow} and \eqref{eq:dual-mdr} correspond
to \eqref{eq:dual-cont-force}, \eqref{eq:dual-cont-flow} and
\eqref{eq:dual-cont-mdr}.

The multiplier
$\lambda_h^{m+1}$ and \eqref{eq:dual-constr} give $v_h^{m+1}(q)\cdot M_h(q)=0$
at every active substrate vertex $q$, the discrete form of the attachment
condition built into $\mathcal V_{\Gamma\cup\Gamma_{\rm base}}$.  The
multipliers $\kappa_h^{m+1}$, $\rho_h^{m+1}$ and \eqref{eq:dual-z} give
$z_h^{m+1}\in\mathcal T_h^m$.
The scheme is defined for any admissible $N_h$ and $M_h$.  We choose them as for the level-set BGN motion, $N_h=\hat n_h^m$ and $M_h=M_h^{(k)}$ in every solve of Algorithm~\ref{alg:picard}, so that both schemes use the same normal and attachment fields.

\begin{theorem}[Dual-MDR full-substrate scheme]
\label{thm:dual}
Assume the hypotheses of Lemma~\ref{lem:universal}.  Then:
\begin{enumerate}[label={\rm(\roman*)}]
\item the scheme \eqref{eq:dual} has a unique solution;
\item the solution satisfies
$W_h^{m+1}+\tau\beta\|\nabla_{\Gamma_h^m}(\gamma_0H_h^{m+1})\|_{L^2(\Gamma_h^m)}^2\le
W_h^m$;
\item if Algorithm~\ref{alg:picard} uses \eqref{eq:dual} in every solve,
including the initial guess, with the same frozen field $M_h^{(k)}$ in all
pairings and with $N_h$ and $M_h^{(k)}$ satisfying these
hypotheses, then every iterate satisfies (ii), and a fixed point satisfies
$\phi(X_h^{m+1}(q))=\phi(q)$ at every substrate vertex.
\end{enumerate}
\end{theorem}

\begin{proof}
\emph{(i)}  It
suffices to show that the homogeneous system, in which
$\mathcal A_h(X_h^{m+1},w_h)$ is replaced by $\tau\mathcal A_h(v_h^{m+1},w_h)$,
has only the trivial solution.  We test it with
\[
w_h=v_h^{m+1},\ \ \varphi_h=\gamma_0H_h^{m+1},\ \ \psi_h=\lambda_h^{m+1},\ \
\eta_h=z_h^{m+1},\ \ \chi_h=\kappa_h^{m+1},\ \ \zeta_h=\rho_h^{m+1} .
\]
The multiplier terms cancel, and \eqref{eq:dual-mdr} with \eqref{eq:dual-z}
gives $\mathcal D_h(z_h^{m+1},v_h^{m+1})=0$.  Hence
\[
\tau\mathcal A_h(v_h^{m+1},v_h^{m+1})
+\beta\|\nabla_{\Gamma_h^m}(\gamma_0H_h^{m+1})\|_{L^2(\Gamma_h^m)}^2=0 .
\]
Both terms are nonnegative, so $H_h^{m+1}$ is constant on
{the connected surface} $\Gamma_h^m$, and \eqref{eq:dual-flow}, \eqref{eq:dual-constr} give
$v_h^{m+1}\in\mathcal T_h^m$.  Testing \eqref{eq:dual-mdr} with $v_h^{m+1}$
gives $\mathcal D_h(v_h^{m+1},v_h^{m+1})=0$, so $v_h^{m+1}=0$ because
$\mathcal D_h$ is positive definite on $\mathcal T_h^m$.  By \eqref{eq:dual-z},
$z_h^{m+1}\in\mathcal T_h^m$, and testing the homogeneous \eqref{eq:dual-force} with
$z_h^{m+1}$ gives $\mathcal D_h(z_h^{m+1},z_h^{m+1})=0$, so $z_h^{m+1}=0$.

Testing the remaining equations with nodal basis functions as in the proof of Theorem~\ref{thm:wellposed-neg} gives nodal identities with positive lumped weights.  Since
$N_h(q)\ne0$, $M_h(q)\ne0$, and $N_h(q)$, $M_h(q)$ are linearly independent at
every contact-line vertex, they give
$H_h^{m+1}=\lambda_h^{m+1}=\kappa_h^{m+1}=\rho_h^{m+1}=0$.

\emph{(ii)}  Test \eqref{eq:dual-force} with $w_h=v_h^{m+1}$,
\eqref{eq:dual-flow} with $\varphi_h=\gamma_0H_h^{m+1}$, \eqref{eq:dual-constr}
with $\psi_h=\lambda_h^{m+1}$, and \eqref{eq:dual-mdr} with $\eta_h=z_h^{m+1}$.  By
\eqref{eq:dual-constr} the multiplier term vanishes, by \eqref{eq:dual-mdr}, \eqref{eq:dual-z} and the symmetry of $\mathcal D_h$ the dual term vanishes, and
$\mathcal A_h(X_h^{m+1},v_h^{m+1})
=-\beta\|\nabla_{\Gamma_h^m}(\gamma_0H_h^{m+1})\|_{L^2(\Gamma_h^m)}^2$.  Since
$X_h^{m+1}-\id=\tau v_h^{m+1}$, applying the area inequality \eqref{eq:area-ineq-neg} on each patch gives
\[
W_h^{m+1}-W_h^m\le\mathcal A_h(X_h^{m+1},X_h^{m+1}-\id)
=-\tau\beta\|\nabla_{\Gamma_h^m}(\gamma_0H_h^{m+1})\|_{L^2(\Gamma_h^m)}^2 .
\]

\emph{(iii)}  The proof of (ii) does not use the values of
$N_h$ and $M_h$, so it applies to every iterate.  At a fixed point the
identity \eqref{eq:discrete-chain-rule} gives
$\phi(X_h^{m+1}(q))=\phi(q)$, as in Theorem~\ref{theorem:exact}.
\end{proof}

\begin{remark}
\label{rem:dual-cases}
The proof of Theorem~\ref{thm:dual} uses the stiffness forms only
through their nonnegativity and the area inequality, so it carries over to the
following cases.
\begin{enumerate}[label={\rm(\alph*)},leftmargin=*,itemsep=1pt,topsep=2pt]
\item \emph{Single-patch schemes.}  Restrict \(\mathcal A_h\) and
\(\mathcal D_h\) to \(\Gamma_h^m\) and the patch \(\Gamma_{\star,h}^m\) of
\eqref{eq:sigma}, with weights \(\gamma_0\) and \(\sigma\) in \(\mathcal A_h\);
the energy is \(W_{\star,h}^m\) of \eqref{eq:single-energy}.
\item \emph{Free surface only.}  If \(\gamma_{\rm in}=\gamma_{\rm out}\),
mesh only \(\Gamma_h^m\) and impose the attachment at the contact-line
vertices; the energy is \(\gamma_0|\Gamma_h^m|\) (Remark~\ref{rem:equal-densities}).
\item \emph{Anisotropic energies.}  Add
\(-\mathcal D_h(z_h^{m+1},w_h)\) to \eqref{eq:aniso-scheme-b}, together with
\eqref{eq:dual-mdr} and \eqref{eq:dual-z}, keeping the isotropic form
\(\mathcal D_h\).  Lemma~\ref{lem:aniso-area} replaces the area inequality,
the energy is \(W_h^{\gamma,m}\), and the dissipation is
\(\tau\|\beta(n_h^m)^{1/2}\nabla_{\Gamma_h^m}\varkappa_{\gamma,h}^{m+1}\|_{L^2(\Gamma_h^m)}^2\).
If \(\gamma_{\rm in}=\gamma_{\rm out}\), only the free surface is meshed, as
in (b).
\end{enumerate}
\end{remark}
}
\section{Numerical experiments}
\label{sec:numerics}
\raggedbottom

The computations for surfaces in 3d use NGSolve/Netgen
\cite{Schoberl1997,Schoberl2014} and a sparse direct solver. The presented schemes easily
generalize to the evolution of curves in 2d, and the corresponding simulations
use an ALBERTA
\cite{SchmidtSiebert2005alberta} implementation with iterative linear solvers.
The substrate is \(\{\phi=0\}\), the time step is \(\tau=T/N_T\), and \(h\) denotes the mesh size.  We take \(\beta\equiv1\) in \eqref{eq:sd} and \eqref{eq:aniso-sd}.

{For nodal attachment we use Algorithm~\ref{alg:picard} with the
difference-quotient map \(\mathcal M_q^{\rm dq}\) of \eqref{eq:rank-one} and
\({\rm tol}_\phi=10^{-10}\). {In all computations \(N_h=\hat n_h^m\).} {The average number of Picard iterations per time step is at most \(5\) in all runs.}  For comparison, we also use the level-set normal, that is, we stop after the first solve, in which \(M_h(q)=\nabla\phi(q)=\mathcal M_q^{\rm dq}(0)\);
by Remark~\ref{rem:non-normalized} this is the same as using the unit normal
\eqref{eq:level-set-normal}.}

We record the attachment residual per step and the accumulated drift,
\[
r_\phi^m:=\max_q|\phi(q+\tau v_h^{m+1}(q))-\phi(q)|,
\qquad
{d_\phi^m:=\max_q|\phi(q)|},
\]
where \(q\) runs over the {vertices of \(\Gamma_{{\rm base},h}^m\)}.  We also monitor \(\max R/r\), the largest ratio of circumradius to inradius over all triangles, and the energy \(W_h^m\).  In all runs the discrete energy is monotonically decreasing, and every Picard
iterate satisfies \(W_h^{(k)}\le W_h^m\).

\subsection{Curves in two dimensions}
\label{subsec:num-curves}

We start with curves in $\mathbb R^2$, where the contact line becomes two
contact points.  In all runs the substrate curve has fixed ends at $x=\pm2$,
$\gamma_0=\gamma_{\rm out}=1$, and the film is a single drop.
We use the full-substrate scheme, in which the film, the wetted part and the
two dry parts of the substrate curve are all discretized, together with nodal
attachment by Algorithm~\ref{alg:picard}.

\begin{figure}[!ht]
  \centering
  \begin{subfigure}[t]{0.32\textwidth}\centering
    \includegraphics[width=\linewidth]{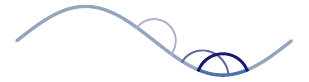}
    \par\smallskip{\small (a) $t=0,0.5,1$\par}
  \end{subfigure}\hfill
  \begin{subfigure}[t]{0.32\textwidth}\centering
    \includegraphics[width=\linewidth]{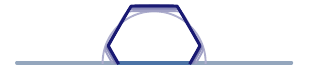}
    \par\smallskip{\small (b) $t=0,\dots,0.1$\par}
  \end{subfigure}\hfill
  \begin{subfigure}[t]{0.32\textwidth}\centering
    \includegraphics[width=\linewidth]{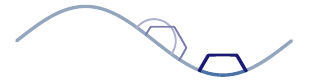}
    \par\smallskip{\small (c) $t=0,0.1,0.5,1$\par}
  \end{subfigure}
  \caption{Drops in two dimensions; later times are darker.  (a) Isotropic
  drop sliding into the valley; (b) anisotropic drop ($L=3$) relaxing to the
  Winterbottom shape; (c) anisotropic drop ($L=3$) sliding into the valley.}
  \label{fig:curves}
\end{figure}

\emph{Isotropic drop on a sine-shaped substrate.}  The substrate is
$y=-\frac12\sin(\frac{\pi x}{2})$ and $\gamma_{\rm in}=1$, so the Young angle
is $90^\circ$.  The drop starts as a circular arc on the flank.  With nodal
attachment, $250$ vertices, $\tau=10^{-3}$ and $T=1$, it slides into the
valley at $x=1$ and comes to rest there, with contact angles
$87.8^\circ$ and $88.8^\circ$ (Figure~\ref{fig:curves}(a)).  The vertices
stay on the substrate to $3.6\times10^{-9}$, and the energy decreases
monotonically.

\emph{Anisotropic drops.}  We use the scheme of Section~\ref{sec:anisotropic}.  The free-surface density is the planar version of
\eqref{eq:aniso-gamma} with $L=3$,
$a_\ell=(\cos\frac{\ell\pi}{3},\sin\frac{\ell\pi}{3})$ and
$\underline G_\ell=a_\ell\otimes a_\ell+\epsilon^2(I-a_\ell\otimes a_\ell)$,
$\epsilon=10^{-2}$.  Its Wulff shape, the curve of least energy for a given
area, is nearly a hexagon.

On the flat substrate $y=0$ we take $\gamma_{\rm in}=1.75$ and start from a
half-disk of radius $0.75$, with $400$ vertices, $\tau=10^{-3}$ and $T=0.1$.
The drop relaxes to the Winterbottom shape \cite{Winterbottom1967,BaoJiangSrolovitzWang2017,KreutzSchmidt2026}, the Wulff shape cut by the
substrate (Figure~\ref{fig:curves}(b)).  The ratio of its height to its
maximal width is $0.624$, against $0.623$ from the Winterbottom construction.
On the sine-shaped substrate we take $\gamma_{\rm in}=1$, $\tau=5\times10^{-3}$
and the initial arc of (a).  The drop becomes hexagonal and slides into the
valley (Figure~\ref{fig:curves}(c)).
In both runs the energy decreases monotonically.

\subsection{Nodal attachment versus the level-set normal on a sphere}
\label{subsec:num-bowl}

{The substrate is the unit sphere centred at \((0,0,-\tfrac12)\).  The initial film is the upper hemisphere of
radius \(\sqrt3/2\) centred at the origin.  The wetted patch \(\Gamma_{\rm in}\) is the cap of the sphere above \(z=0\), so the initial contact angle is \(30^\circ\).  We use the single-patch
scheme with \(h=1/16\) and \(T=0.1\), and compare nodal attachment with the level-set
normal, that is, Algorithm~\ref{alg:picard} with and without the Picard
iteration.

First let \(\gamma_0=\gamma_{\rm out}=1\) and \(\gamma_{\rm in}=1/2\), so that the Young angle is \(\theta=60^\circ\) and the relaxation is mild.  Since \(\gamma_{\rm in}<\gamma_{\rm out}\), the scheme meshes the dry patch \(\Gamma_{\rm out}\), the part of the sphere below \(z=0\) (\(\star={\rm out}\)).  With \(\tau=5\times10^{-4}\) both runs dissipate energy, but only with nodal attachment do the vertices stay on the sphere: \(d_\phi^N=1.27\times10^{-8}\) against \(9.05\times10^{-3}\), and \(\max R/r=2.89\) against \(20.7\) (Table~\ref{tab:bowl} and Figure~\ref{fig:bowl-mild}).

Next let \(\gamma_0=\gamma_{\rm in}=1\) and \(\gamma_{\rm out}=1/2\), so that \(\theta=120^\circ\).  The drop therefore contracts strongly; we take \(\star={\rm in}\) and \(\tau=5\times10^{-5}\).
{Figure~\ref{fig:bowl-attachment} shows the drop in full (top) and with the
front half of the free surface removed (bottom).  With nodal attachment
(Figure~\ref{fig:bowl-attachment}(b),(e)) the drop contracts and rises on the
sphere, and the wetted patch stays on the sphere.
The drift is \(d_\phi^N=1.5\times10^{-8}\), and \(\max R/r\le3.75\) (Table~\ref{tab:bowl}).
With the level-set normal
(Figure~\ref{fig:bowl-attachment}(c),(f)) the substrate vertices leave the
sphere already in the first step and do not return. The computation stops at
\(t=0.036\) with \(\max R/r>10^4\). }}

\begin{table}[htbp]
  \centering
  \caption{Diagnostics for nodal attachment and the level-set normal.}
  \label{tab:bowl}
{\begin{tabular}{llccc}
  \hline
  \(\theta\) & normal & \(\max_m r_\phi^m\) & \(d_\phi^N\) & \(\max R/r\) \\
  \hline
  \(60^\circ\) & nodal attachment & \(9.87\times10^{-11}\) & \(1.27\times10^{-8}\) & \(2.89\) \\
  & level-set normal & \(8.61\times10^{-3}\) & \(9.05\times10^{-3}\) & \(20.7\) \\
  \(120^\circ\) & nodal attachment & \(9.98\times10^{-11}\) & \(1.47\times10^{-8}\) & \(3.75\) \\
  & level-set normal & \(4.03\times10^{-2}\) & \(4.08\times10^{-2}\) & \(>10^4\) \\
  \hline
  \end{tabular}}
\end{table}

\begin{figure}[!htb]
  \centering
  \begin{subfigure}[t]{0.23\textwidth}
    \centering
    \includegraphics[width=\linewidth]{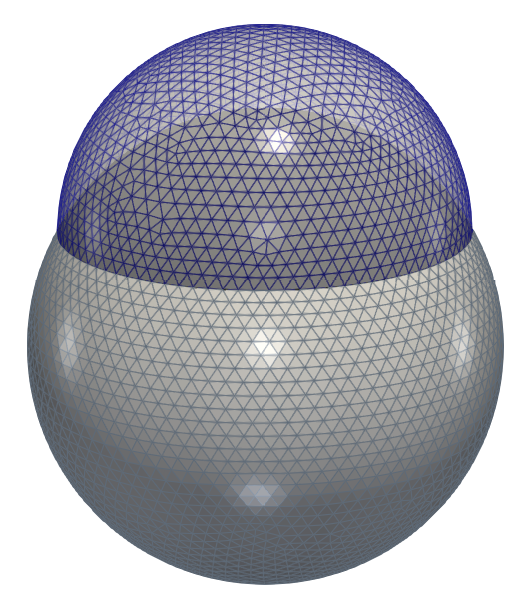}
    \par\smallskip{\small (a) initial, \(t=0\)\par}
  \end{subfigure}\hfill
  \begin{subfigure}[t]{0.23\textwidth}
    \centering
    \includegraphics[width=\linewidth]{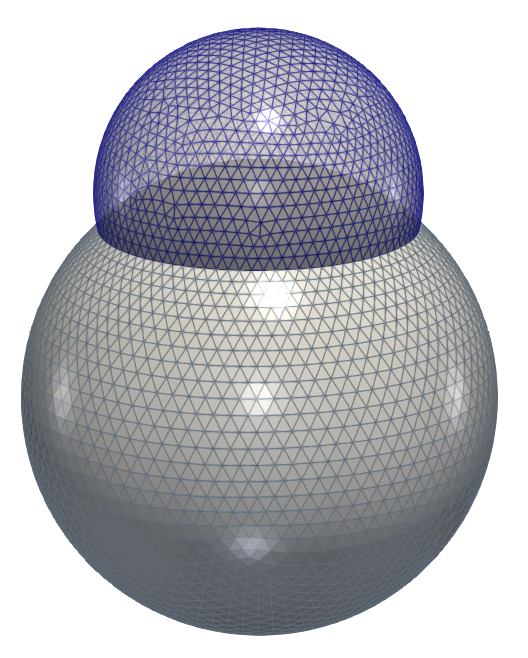}
    \par\smallskip{\small (b) nodal attachment, \(t=0.1\)\par}
  \end{subfigure}\hfill
  \begin{subfigure}[t]{0.23\textwidth}
    \centering
    \includegraphics[width=\linewidth]{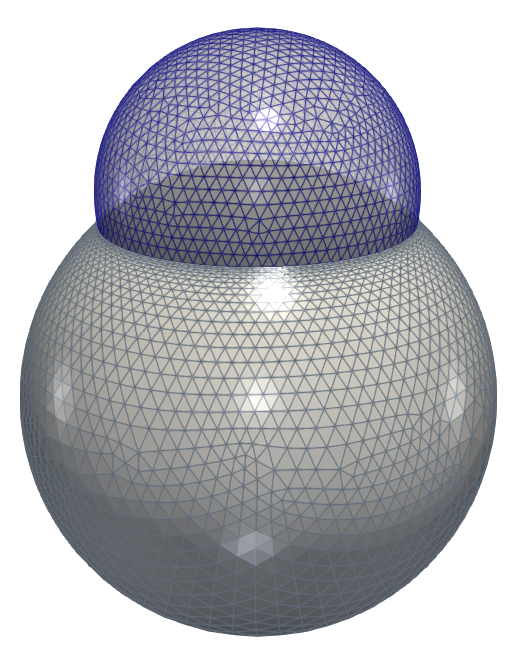}
    \par\smallskip{\small (c) level-set normal, \(t=0.1\)\par}
  \end{subfigure}
  \caption{\(\theta=60^\circ\).  The lower surface is the meshed dry patch.}
  \label{fig:bowl-mild}
\end{figure}

\begin{figure}[!htb]
  \centering
  \begin{subfigure}[t]{0.23\textwidth}
    \centering
    \includegraphics[width=\linewidth]{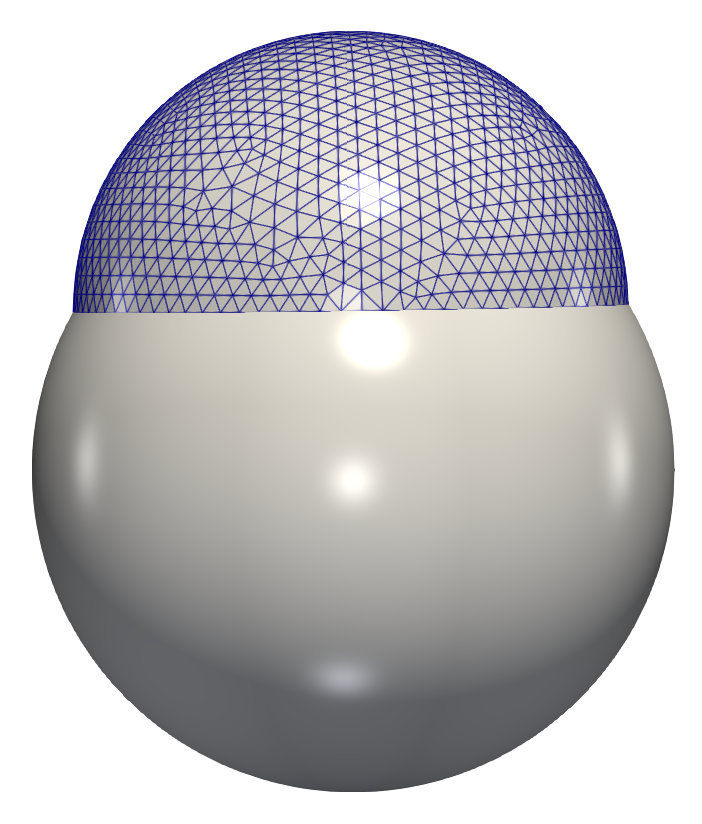}
    \par\smallskip{\small (a) initial, \(t=0\)\par}
  \end{subfigure}\hfill
  \begin{subfigure}[t]{0.23\textwidth}
    \centering
    \includegraphics[width=\linewidth]{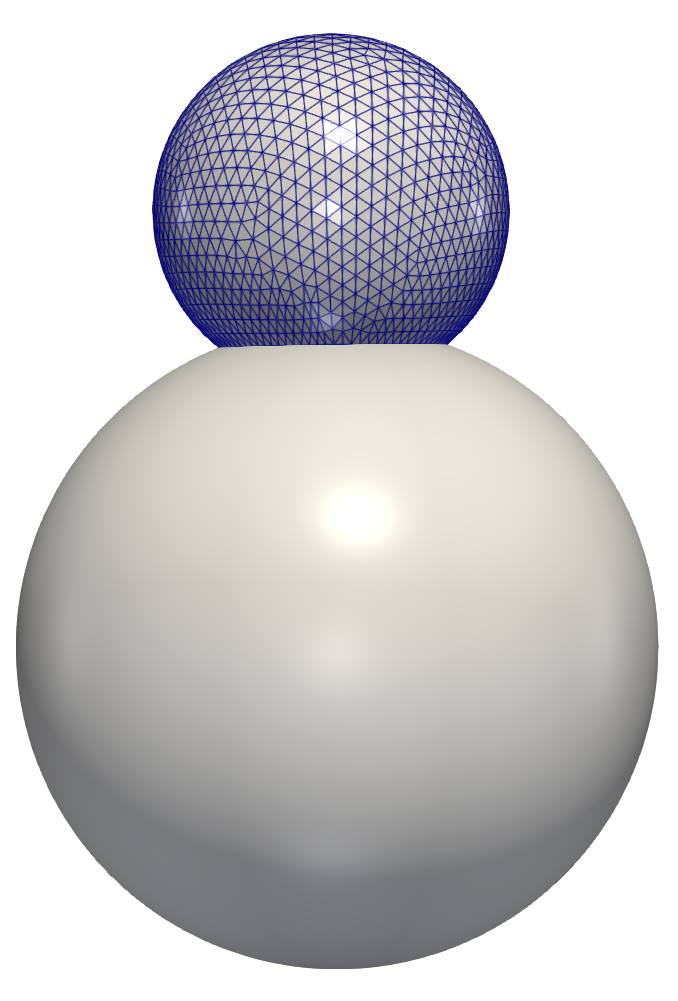}
    \par\smallskip{\small (b) nodal attachment, \(t=0.1\)\par}
  \end{subfigure}\hfill
  \begin{subfigure}[t]{0.23\textwidth}
    \centering
    \includegraphics[width=\linewidth]{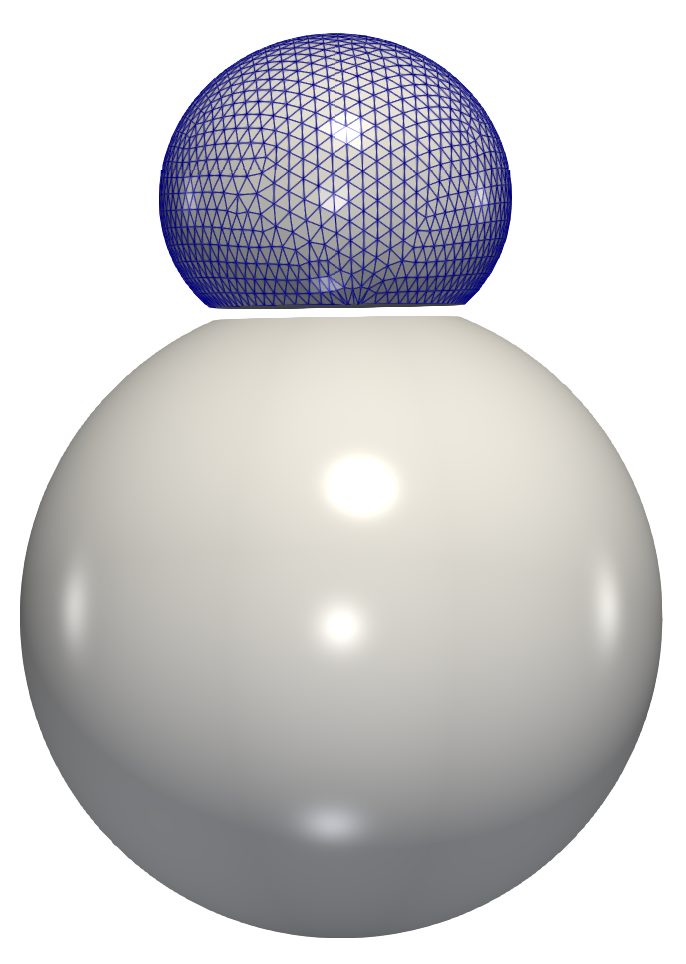}
    \par\smallskip{\small (c) level-set normal, \(t=0.03\)\par}
  \end{subfigure}\\[4pt]
  \begin{subfigure}[t]{0.23\textwidth}
    \centering
    \includegraphics[width=\linewidth]{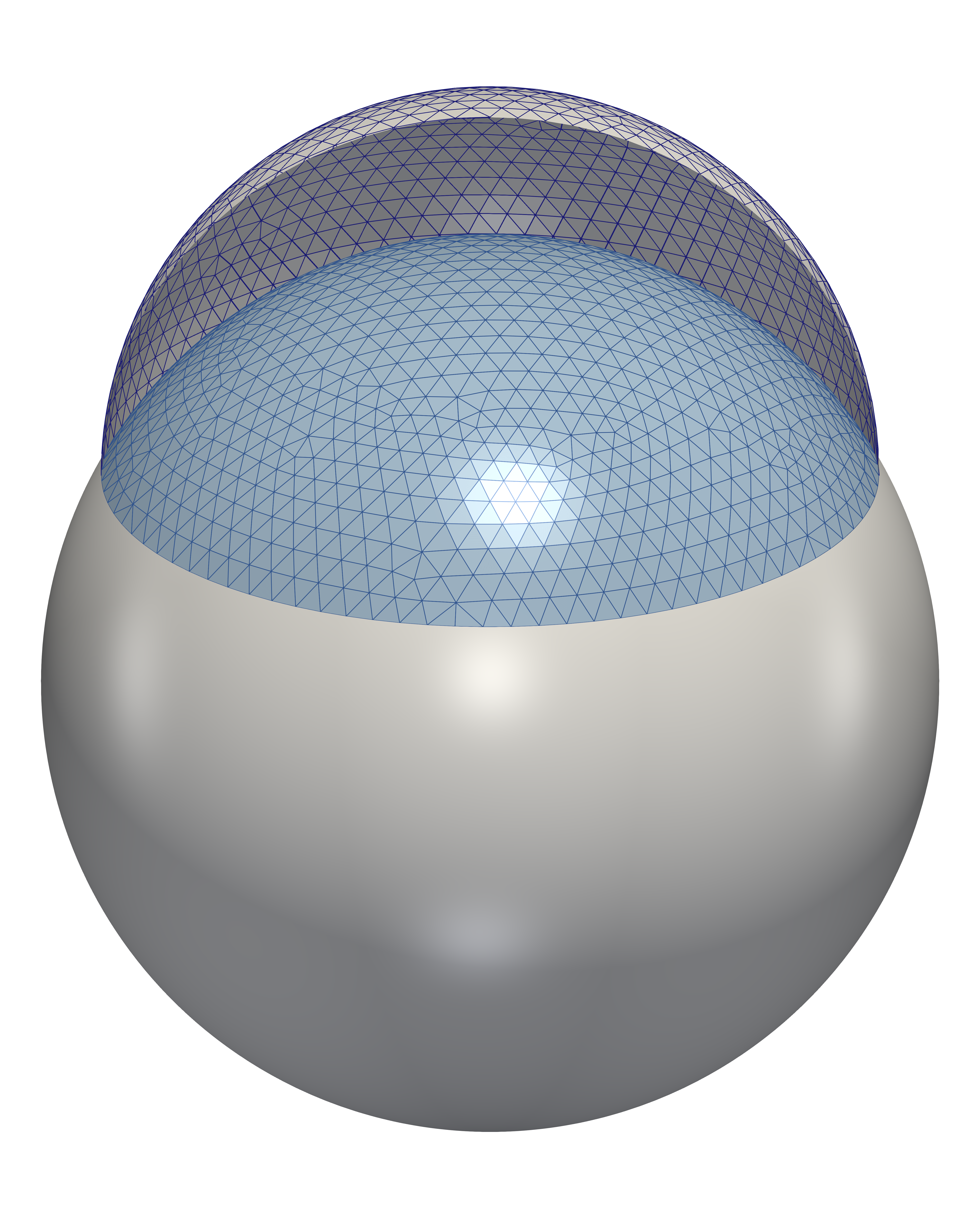}
    \par\smallskip{\small (d) as (a)\par}
  \end{subfigure}\hfill
  \begin{subfigure}[t]{0.23\textwidth}
    \centering
    \includegraphics[width=\linewidth]{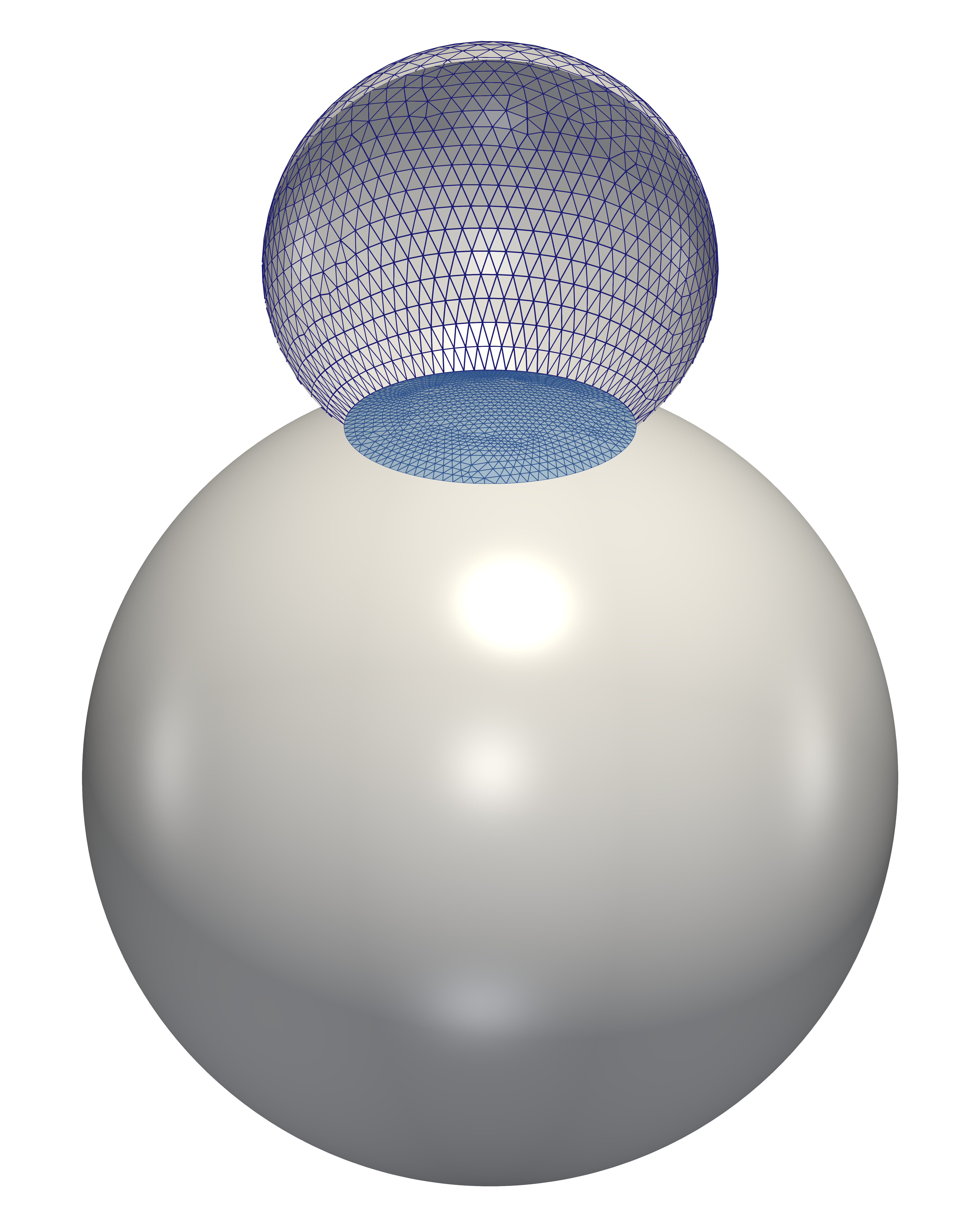}
    \par\smallskip{\small (e) as (b)\par}
  \end{subfigure}\hfill
  \begin{subfigure}[t]{0.23\textwidth}
    \centering
    \includegraphics[width=\linewidth]{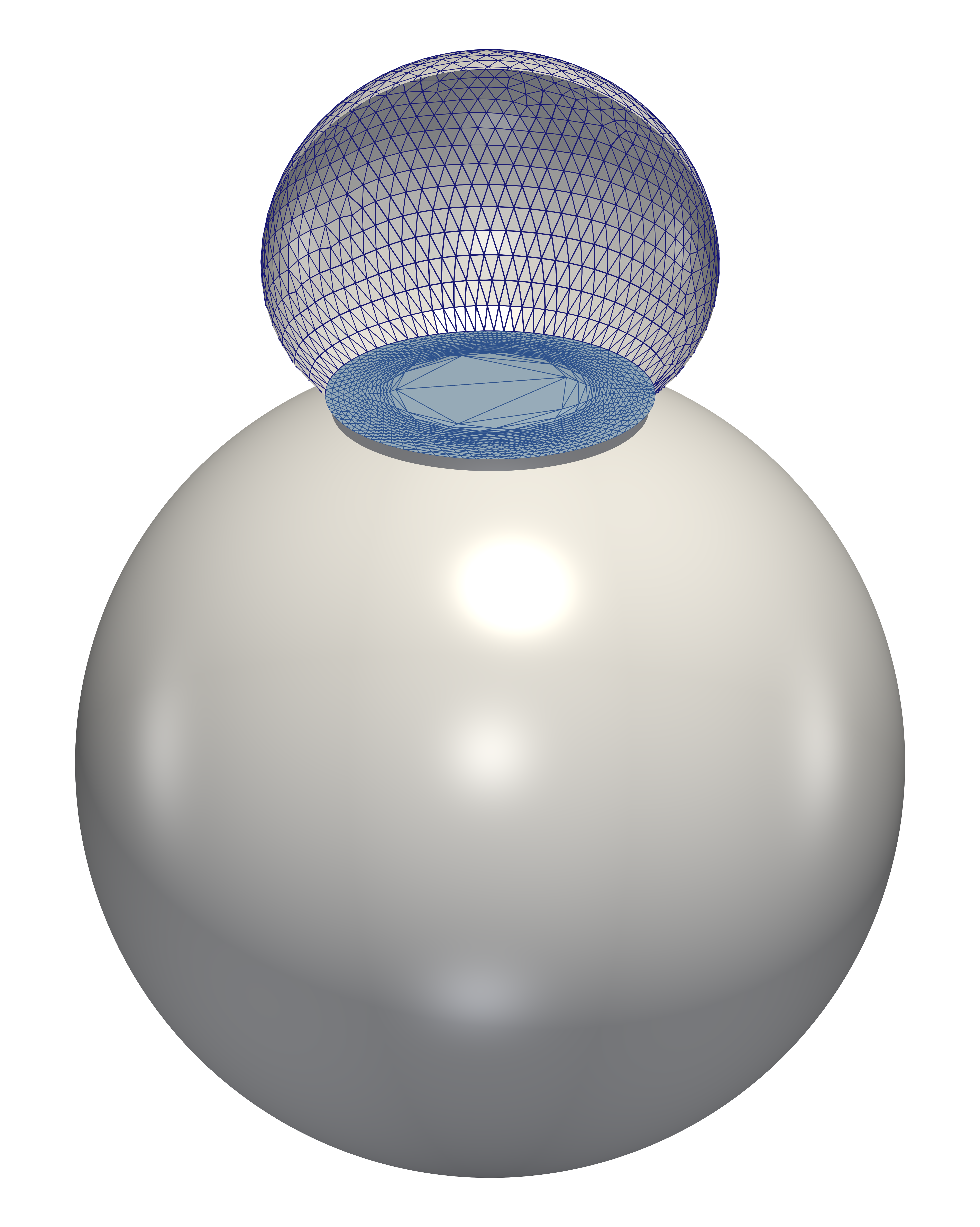}
    \par\smallskip{\small (f) as (c)\par}
  \end{subfigure}
  \caption{{\(\theta=120^\circ\). Top: full view.  Bottom: the front half
  of the free surface is removed to show the wetted patch (light blue) on the
  sphere.}}
  \label{fig:bowl-attachment}
\end{figure}

\subsection{Curvature-driven migration on a non-axisymmetric substrate}
\label{subsec:num-egg}

The substrate is the egg-carton graph
\[
z=f(x,y)=0.06\,\cos\bigl(3(x-\tfrac\pi6)\bigr)\,\cos(3y),
\qquad \phi=f-z,
\]
over the disk of radius \(2.5\).  It has a peak/valley at
\((\pm\frac\pi6,0)\).  For the normal pointing away
from the substrate, its mean curvature is largest at the peak and smallest in
the valley.  The drop starts at the origin.  We take \(\gamma_0=\gamma_{\rm in}=1\) and
\(\gamma_{\rm out}=1/2\), so that the Young angle is \(120^\circ\).  To build
the initial meshes, we start from a flat disk of radius \(2.5\), split into a
wetted disk of radius \(0.35\) and the surrounding annulus, put a
\(120^\circ\) spherical cap on the wetted disk, and lift all three meshes by
\((x,y,z)\mapsto(x,y,z+f(x,y))\).  The lifted meshes are conforming, every
substrate vertex lies on the substrate, and the contact angle becomes
\(116.9^\circ\pm1.7^\circ\).  We use the full-substrate scheme with
nodal attachment by the Picard iteration of Algorithm~\ref{alg:picard},
\(h=1/12\), \(\tau=h^2\), and
\(T=2\).

\begin{figure}[htbp]
  \centering
  \begin{minipage}[c]{0.33\textwidth}\centering
    \includegraphics[width=\linewidth]{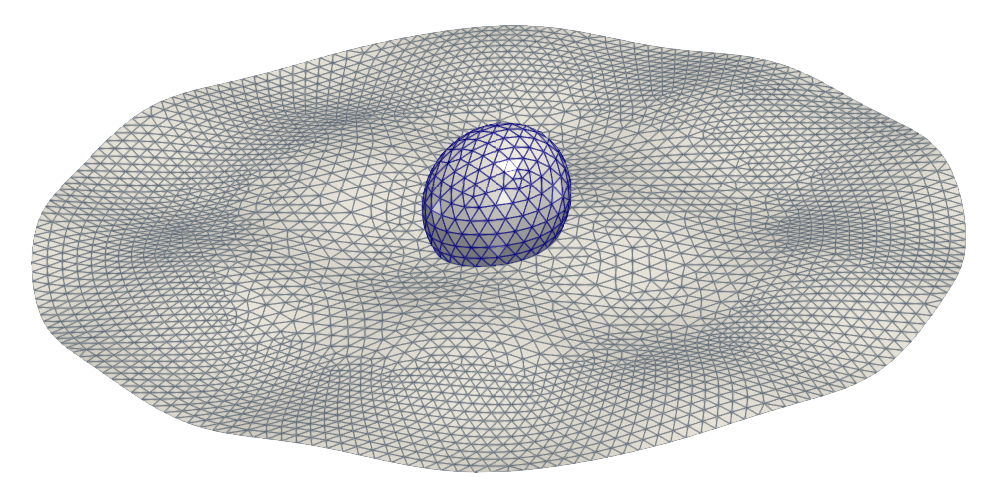}
  \end{minipage}\hfill
    \begin{minipage}[c]{0.33\textwidth}\centering
    \includegraphics[width=\linewidth]{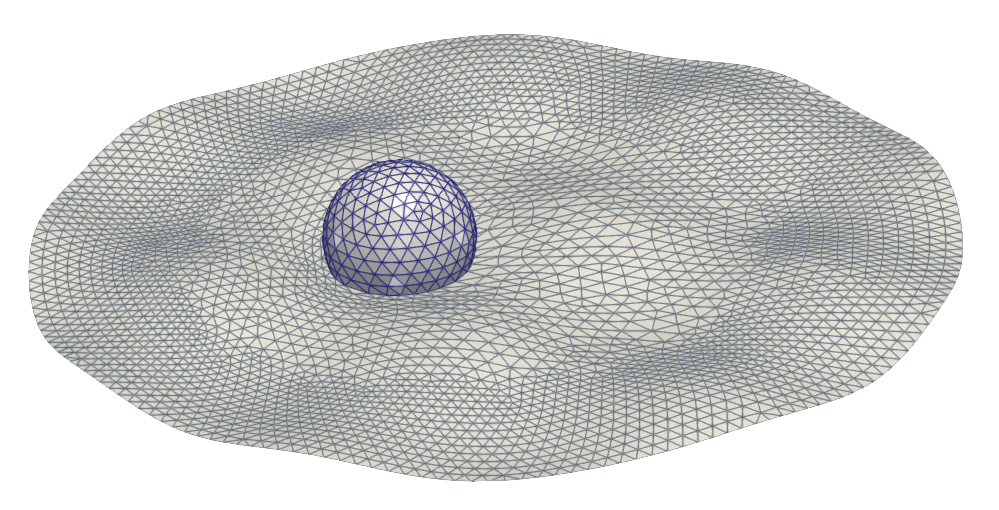}
  \end{minipage}\hfill
  \begin{minipage}[c]{0.33\textwidth}\centering
    \includegraphics[width=\linewidth]{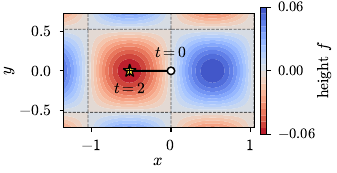}
  \end{minipage}\\[3pt]
  \begin{minipage}[t]{0.33\textwidth}\centering
    {\small (a) initial mesh, \(t=0\)}
  \end{minipage}\hfill
  \begin{minipage}[t]{0.33\textwidth}\centering
    {\small (b) final mesh, \(t=T\)}
  \end{minipage}\hfill
  \begin{minipage}[t]{0.30\textwidth}\centering
    {\small (c) centroid path, \(0\le t\le T\)}
  \end{minipage}
  \caption{Egg-carton migration, \(T=2\).  In (c) the colour shows the
  substrate height; the star marks the centroid at \(T\).}
  \label{fig:egg-meshes}
\end{figure}

{The drop moves towards the curvature minimum (Figure~\ref{fig:egg-meshes}).
We follow the centroid of the contact line, weighted by arc length.  At \(T=2\) it has stopped at
\(x=-0.521\), within \(0.003\) of the valley at \(x=-\frac\pi6\).  This agrees qualitatively with
the law derived in \cite{ZhaoJiangWangSrolovitzQianBao2024} for small islands,
whose velocity is proportional to the negative surface gradient of the
substrate mean curvature.}
{Without remeshing or smoothing, \(\max R/r=4.85\) at \(T\); the energy decreases at every step, and the volume changes by \(0.2\%\).}

\subsection{Young-law convergence on an egg-carton substrate}
\label{subsec:num-young}

We next test Young's law on a second egg-carton graph,
\[
f(x,y)=0.12\cos\bigl(1.5(x-0.35)\bigr)
       \cos\bigl(1.5(y+0.20)\bigr),
\]
over the disk of radius \(1.5\).  A spherical cap
with contact radius \(0.55\) and planar angle \(90^\circ\) is lifted by
\((x,y,z)\mapsto(x,y,z+f(x,y))\).  We take \(\gamma_0=\gamma_{\rm in}=1\) and \(\gamma_{\rm out}=1/2\), so that \(\theta=120^\circ\), giving
an initial mean-angle mismatch of about \(20^\circ\).  The discrete angle is
computed from the conormals adjacent to each contact edge and averaged along
the contact line.  We set \(T=1\), \(\tau=h^2\), and
\(h=1/10,1/12,1/16,1/20\).

The mean angle moves towards \(\theta=120^\circ\).  At \(T=1\), its error
decreases from \(6.63^\circ\) to \(3.40^\circ\) over the four meshes, with observed orders
\(0.91\), \(0.99\), and \(0.96\) (Table~\ref{tab:young}).  {Every run satisfies
\(\max_m r_\phi^m\le9.9\times10^{-11}\) and \(\max R/r\le4.1\).}

\begin{table}[htbp]
  \centering
  \caption{Young-law recovery on the egg-carton substrate.}
  \label{tab:young}
  \begin{tabular}{cccccc}
  \hline
  \(h\) & \(N_T\) & \(\theta_h(0)\) & \(\theta_h(T)\)
        & \(|\theta_h(T)-\theta|\) & EOC \\
  \hline
  \(1/10\) & \(100\) & \(97.73^\circ\)  & \(113.375^\circ\) & \(6.625^\circ\) & ---\\
  \(1/12\) & \(144\) & \(98.78^\circ\)  & \(114.391^\circ\) & \(5.609^\circ\) & \(0.91\) \\
  \(1/16\) & \(256\) & \(100.51^\circ\) & \(115.783^\circ\) & \(4.217^\circ\) & \(0.99\) \\
  \(1/20\) & \(400\) & \(101.84^\circ\) & \(116.597^\circ\) & \(3.403^\circ\) & \(0.96\) \\
  \hline
  \end{tabular}
\end{table}

\subsection{Anisotropic drops on a flat substrate}
\label{subsec:num-aniso}

We test the full-substrate scheme of Section~\ref{sec:anisotropic} in three space
dimensions. {The
free-surface density is \eqref{eq:aniso-gamma} with \(L=4\),
\(\underline G_\ell=a_\ell\otimes a_\ell+\epsilon^2(I-a_\ell\otimes a_\ell)\) and
\(\epsilon=10^{-2}\), where the unit vectors \(a_1,a_2,a_3\) are horizontal at
angles \(0,\frac\pi3,\frac{2\pi}3\) and \(a_4=e_3\).  We take
\(\gamma_{\rm in}=\gamma_{\rm out}=1\).}

On the plane $z=0$ a $60^\circ$ spherical cap of contact
radius $0.62$ sits on the disk of radius $1.5$; all three patches are meshed
with $h=1/12$ ($1365$ vertices), and $\tau=10^{-2}$.  On the flat substrate the condition $v_h(q)\cdot e_3=0$
already keeps the substrate vertices on $\{z=0\}$, so no Picard iteration is
needed: we use the scheme
\eqref{eq:aniso-scheme} with $N_h=\hat n_h^m$ and $M_h=e_3$, that is, with the BGN tangential motion.  The drop relaxes to the Winterbottom shape (Figure~\ref{fig:aniso}),
and the energy decreases at every step (Theorem~\ref{thm:aniso-energy}).

\begin{figure}[htbp]
  \centering
  \begin{subfigure}[b]{0.31\textwidth}\centering
    \includegraphics[width=\linewidth]{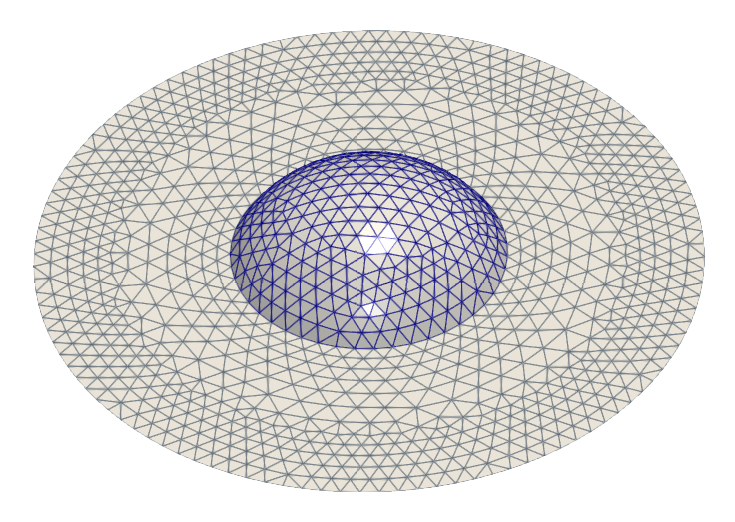}
    \par\smallskip{\small (a) $t=0$\par}
  \end{subfigure}\hfill
  \begin{subfigure}[b]{0.31\textwidth}\centering
    \includegraphics[width=\linewidth]{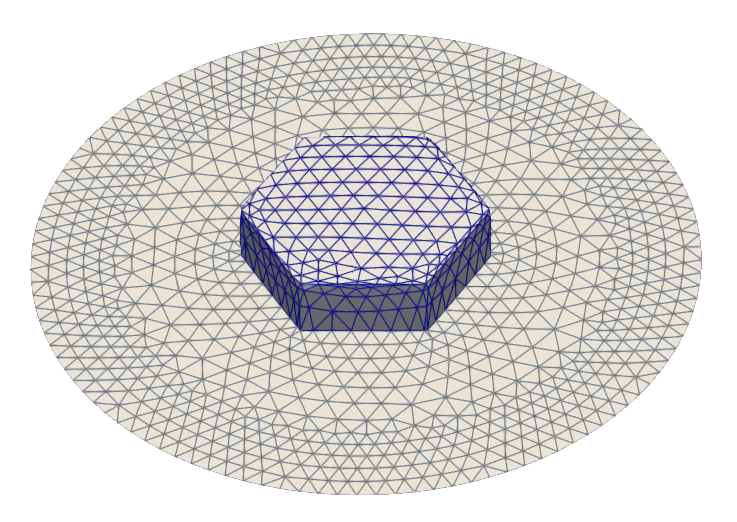}
    \par\smallskip{\small (b) $L=4$, BGN, $t=1$\par}
  \end{subfigure}\hfill
  \begin{subfigure}[b]{0.31\textwidth}\centering
    \includegraphics[width=\linewidth]{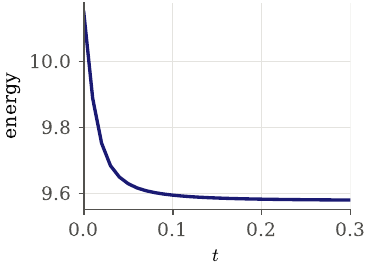}
    \par\smallskip{\small (c) total energy $W_h^{\gamma,m}$\par}
  \end{subfigure}
  \caption{Anisotropic drop on a flat substrate, {BGN method}, $\tau=10^{-2}$.}
  \label{fig:aniso}
\end{figure}

{
\subsection{Tangential motion: level-set BGN versus dual-MDR}
\label{subsec:num-dual}

We compare the level-set BGN motion of \S\ref{subsec:single-patch} with the dual-MDR motion of Section~\ref{sec:dual-mdr} for an isotropic drop (Figure~\ref{fig:dual-valley}); Figures~\ref{fig:dual-film} and~\ref{fig:aniso-sine} show the dual-MDR motion on the free surface alone for an isotropic and an anisotropic drop.  In all runs $N_h=\hat n_h^m$, and $M_h$ is the attachment field of Algorithm~\ref{alg:picard}.  The substrate is $\{\phi=0\}$ with $\phi=z-f(x,y)$, $f(x,y)=-\frac12\sin(\frac{2\pi x}{3})$, over $[-1.5,1.5]^2$.  We take $\gamma_{\rm in}=\gamma_{\rm out}=1$ and, for the isotropic drop, $\gamma_0=1$.  The free surface is initially the half-ellipsoid
\[
\bigl\{(x,y,f(x,y)+0.4s):\ \bigl(\tfrac{x-x_c}{0.4}\bigr)^2
+\bigl(\tfrac{y}{0.8}\bigr)^2+s^2=1,\ s\ge0\bigr\}.
\]
For $x_c=0$ the drop sits on the flank and slides towards the valley at
$x=0.75$ (slope start); for $x_c=0.75$ it starts in the valley and only rounds
off (valley start).  The full-substrate mesh for the isotropic drop has about $1400$ vertices, and the free-surface
meshes have about $240$.  All runs use nodal attachment, and a
run is stopped when a triangle has
$R/r>100$ or reverses its orientation.

\emph{Isotropic drop, valley start.}  The contact line changes from an ellipse with axes $0.8$ and $1.6$ to a nearly circular curve of diameter about $1.06$
(Figure~\ref{fig:dual-valley}).  We use $\tau=2.5\times10^{-4}$.  The level-set BGN run stops after $105$ steps ($t=0.026$),
when substrate triangles near the receding ends of the contact line
degenerate.  With the dual-MDR tangential motion, the mesh stays well shaped up to $T=1$
($\max R/r\le5.3$).

\begin{figure}[htbp]
  \centering
  \setlength{\fboxsep}{0pt}
  \begin{subfigure}[t]{0.32\textwidth}\centering
    \includegraphics[width=\linewidth,height=0.66\linewidth,keepaspectratio]{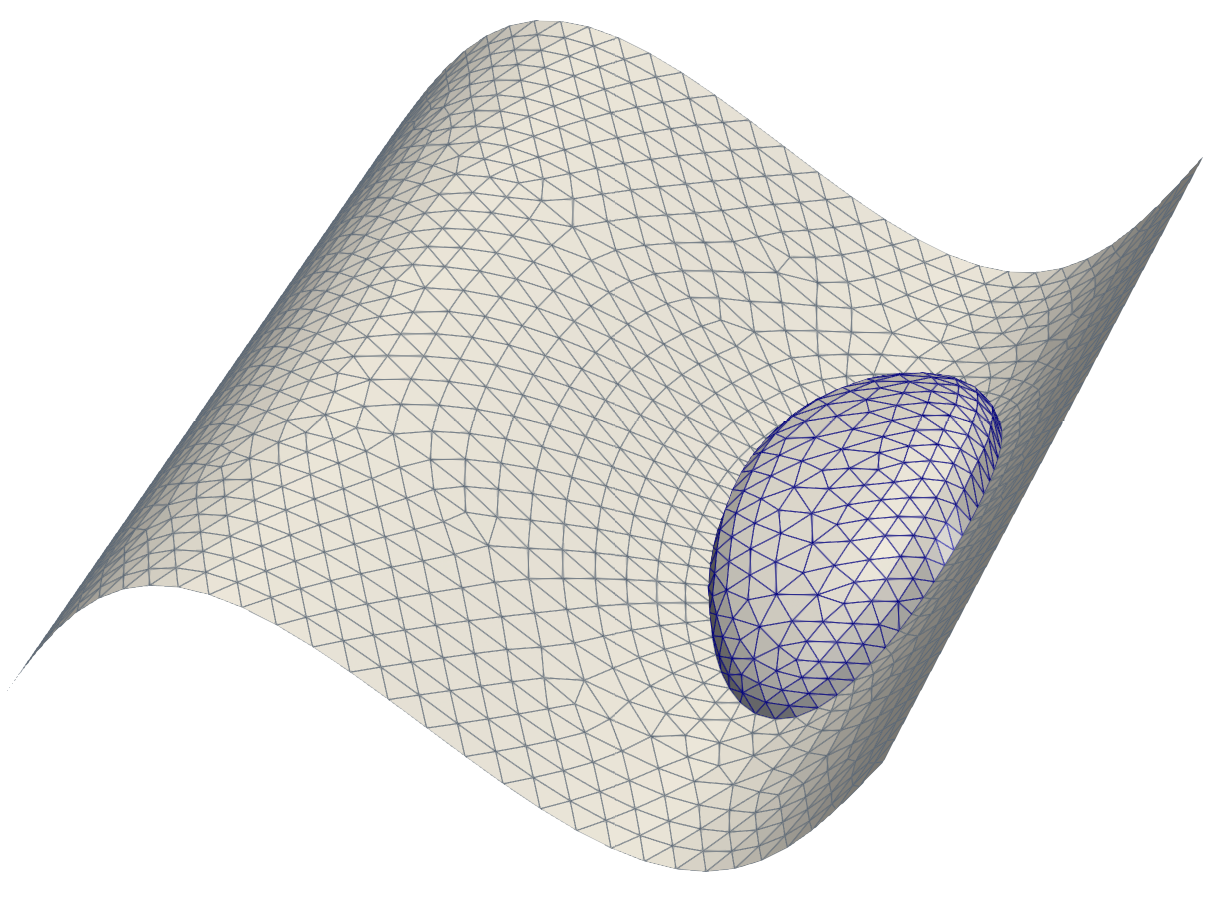}
    \par\smallskip{\small (a) $t=0$\par}
  \end{subfigure}\hfill
  \begin{subfigure}[t]{0.32\textwidth}\centering
    \includegraphics[width=\linewidth,height=0.66\linewidth,keepaspectratio]{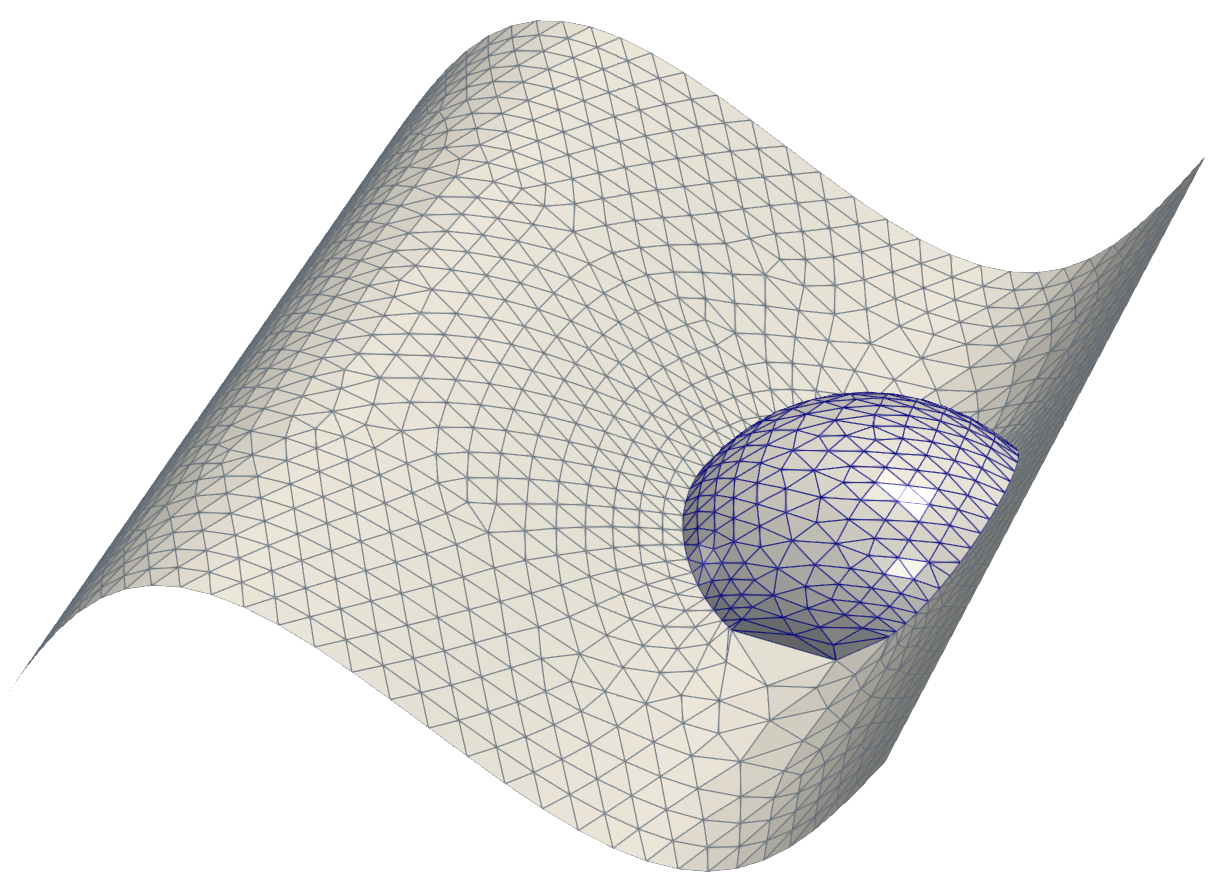}
    \par\smallskip{\small (b) level-set BGN, $t=0.026$\par}
  \end{subfigure}\hfill
  \begin{subfigure}[t]{0.32\textwidth}\centering
    \includegraphics[width=\linewidth,height=0.66\linewidth,keepaspectratio]{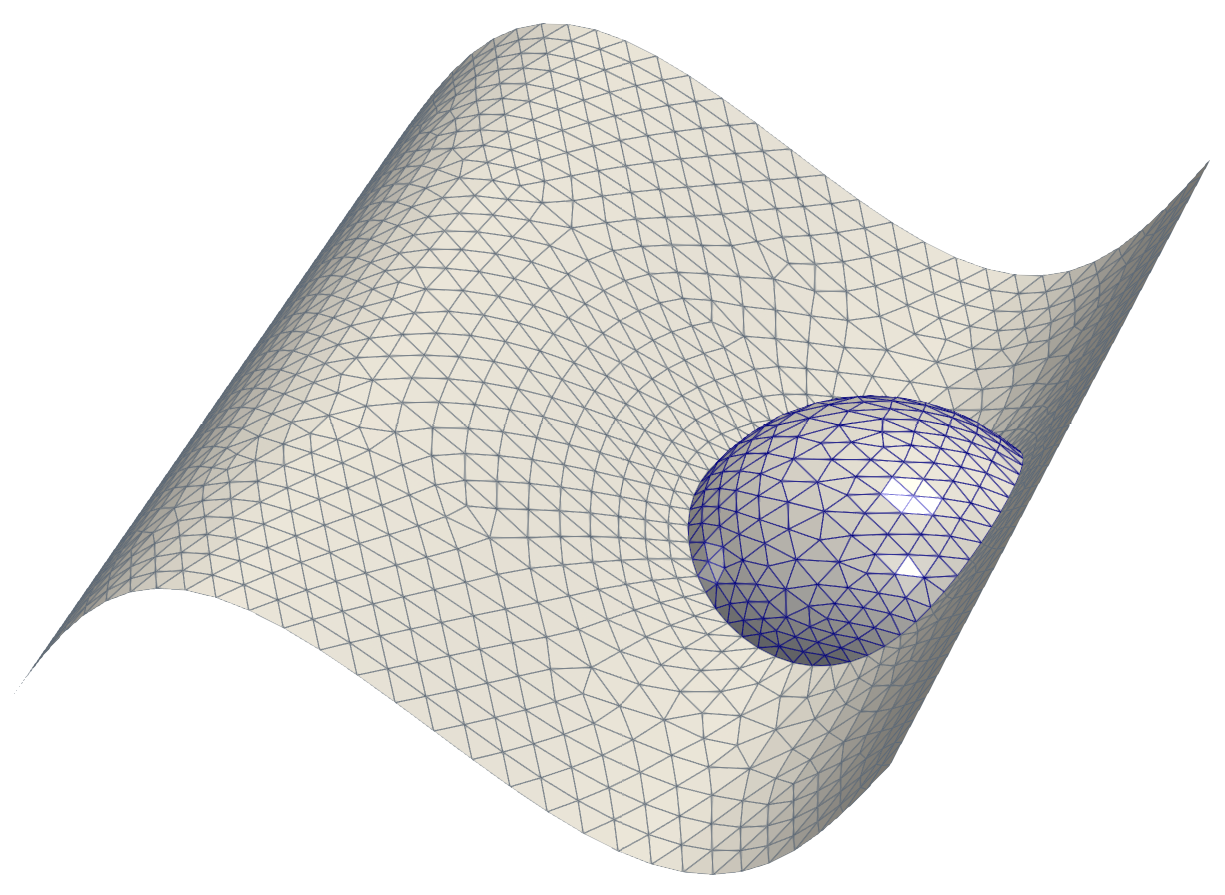}
    \par\smallskip{\small (c) dual-MDR, $t=1$\par}
  \end{subfigure}
  \caption{Valley start, full-substrate scheme with nodal attachment,
  $\tau=2.5\times10^{-4}$.}
  \label{fig:dual-valley}
\end{figure}

\emph{Free surface only.}  Since $\gamma_{\rm in}=\gamma_{\rm out}$, the
substrate part of the energy does not change, and the scheme can be posed on
$\Gamma_h$ alone (Remark~\ref{rem:dual-cases}(b)).  The substrate is then not
meshed; nodal attachment keeps the contact-line vertices on $\{\phi=0\}$.  We
use the slope start and $\tau=2.5\times10^{-4}$ (Figure~\ref{fig:dual-film}).
With the dual-MDR tangential motion, the drop slides into the valley at $x=0.75$, and the mesh
stays well shaped.

\begin{figure}[htbp]
  \centering
  \begin{subfigure}[t]{0.32\textwidth}\centering
    \includegraphics[width=\linewidth,height=0.66\linewidth,keepaspectratio]{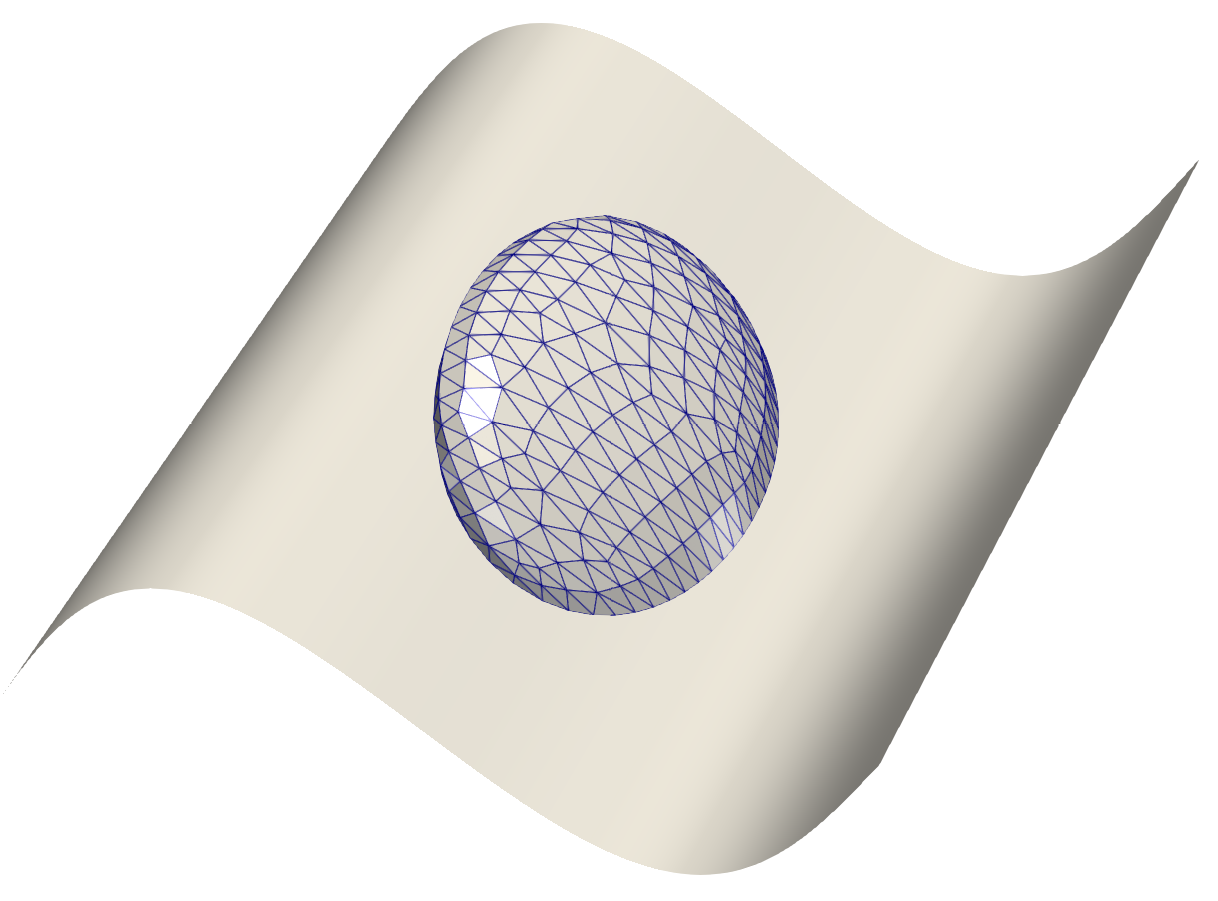}
    \par\smallskip{\small (a) initial state, $t=0$\par}
  \end{subfigure}\hspace{0.02\textwidth}
  \begin{subfigure}[t]{0.32\textwidth}\centering
    \includegraphics[width=\linewidth,height=0.66\linewidth,keepaspectratio]{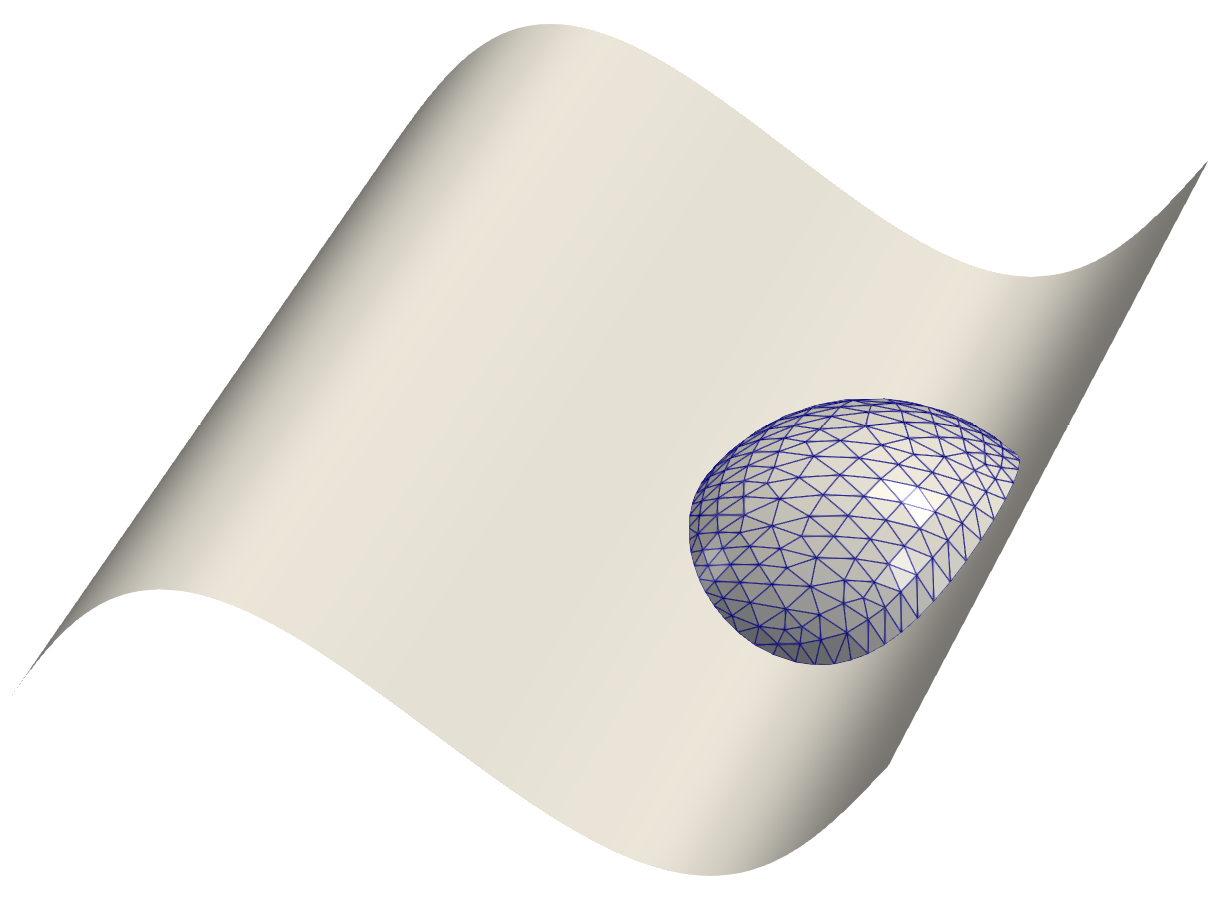}
    \par\smallskip{\small (b) dual-MDR, $t=1$\par}
  \end{subfigure}
  \caption{Slope start, free surface only, dual-MDR,
  $\tau=2.5\times10^{-4}$.}
  \label{fig:dual-film}
\end{figure}

\emph{Anisotropic drop.}  We take the density of
\S\ref{subsec:num-aniso} with $L=3$ and $a_\ell=e_\ell$ (Wulff
shape nearly a cube), start in the valley, and mesh only the free surface
(Remark~\ref{rem:dual-cases}(c)).  With dual-MDR the drop develops
flat facets, the mesh stays regular ($\max R/r\le8.9$), and the
film energy $\int_{\Gamma_h^m}\gamma(n_h^m)$ is monotonically decreasing
(Figure~\ref{fig:aniso-sine}).

\begin{figure}[htbp]
  \centering
  \begin{subfigure}[b]{0.31\textwidth}\centering
    \includegraphics[width=\linewidth]{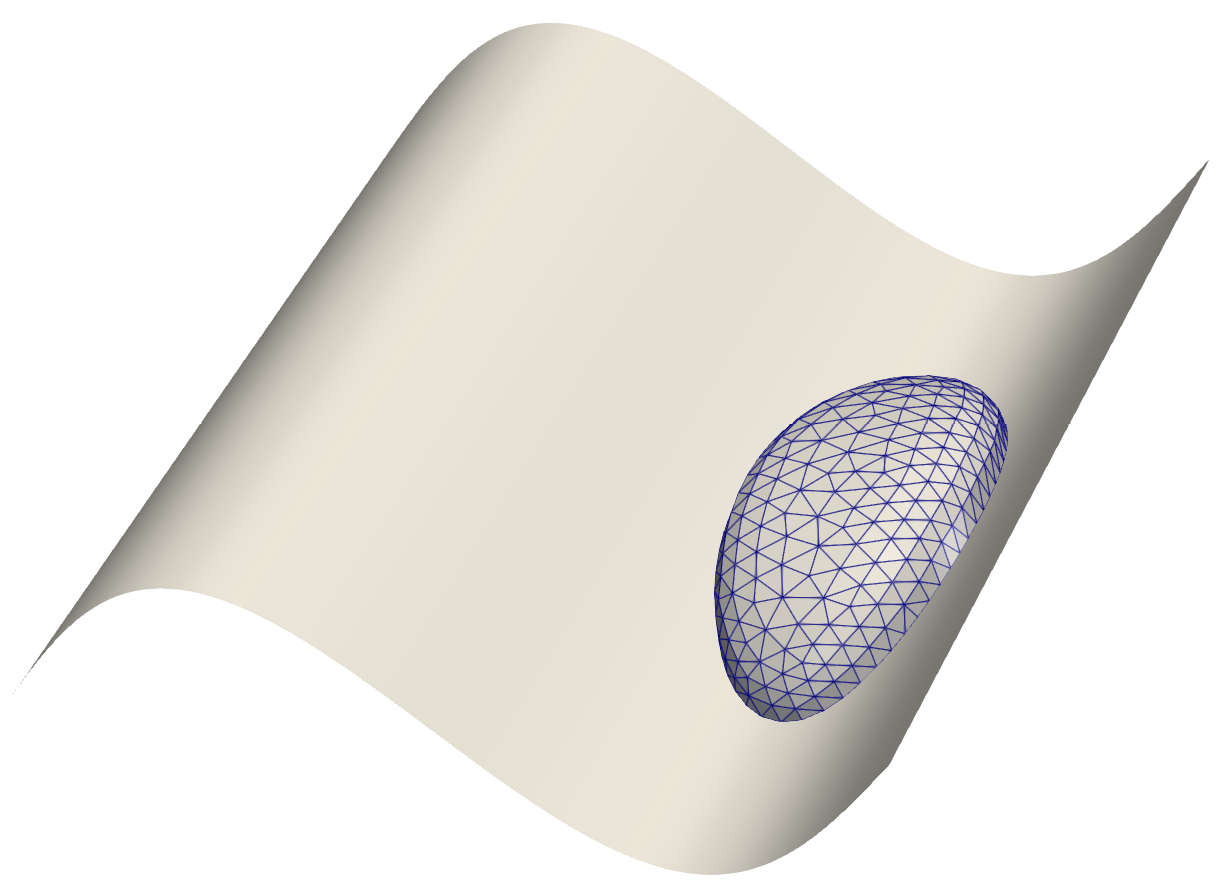}
    \par\smallskip{\small (a) $t=0$\par}
  \end{subfigure}\hfill
  \begin{subfigure}[b]{0.31\textwidth}\centering
    \includegraphics[width=\linewidth]{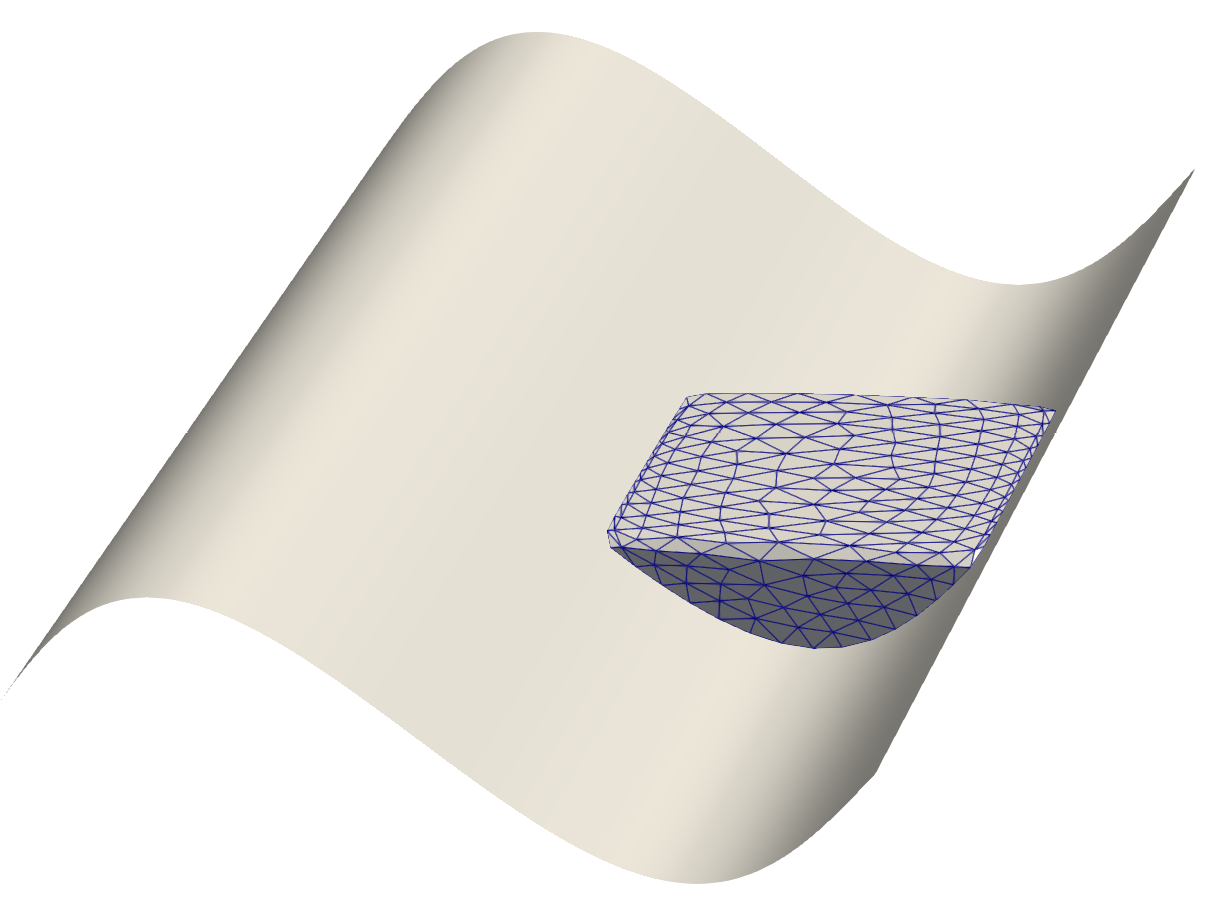}
    \par\smallskip{\small (b) $L=3$, dual-MDR, $t=1$\par}
  \end{subfigure}\hfill
  \begin{subfigure}[b]{0.31\textwidth}\centering
    \includegraphics[width=\linewidth]{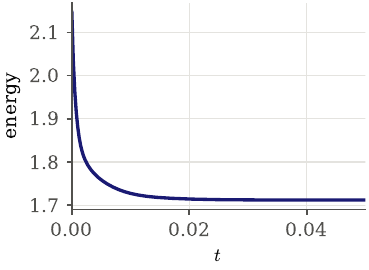}
    \par\smallskip{\small (c) $\int_{\Gamma_h^m}\gamma(n_h^m)$\par}
  \end{subfigure}
  \caption{Anisotropic drop, free
  surface only, dual-MDR,
  $\tau=2.5\times10^{-4}$.  The substrate is drawn as the exact surface.}
  \label{fig:aniso-sine}
\end{figure}

}

{
\subsection{An anisotropic substrate energy}
\label{subsec:num-aniso-sub}

{The substrate is the unit sphere centred at \((0,0,-\frac12)\), and
the whole sphere is meshed.  The film is isotropic with \(\gamma_0=1\),
\(\gamma_{\rm out}=1\), and
\(\gamma_{\rm in}(\vec p)=\frac12\sqrt{\vec p\cdot\underline G\,\vec p}\), which is
\eqref{eq:aniso-gamma} with \(L=1\) and \(\underline G_1=\frac14\underline G\).
For \(\underline G=I\) we have \(\gamma_{\rm in}=\frac12\) and, by Young's law,
\(\theta=60^\circ\); the initial drop is a spherical cap that meets the
substrate at \(60^\circ\), so for \(\underline G=I\) it is a stationary
solution of the continuous problem.  For
\(\underline G={\rm diag}(49,0.01,1)\) wetting costs far more where \(n_b\)
points in the \(x\)-direction than where it points in the \(y\)-direction.  We
use \eqref{eq:aniso-scheme} with the dual-MDR motion of
Remark~\ref{rem:dual-cases}(c), nodal attachment with Picard iteration, \(h=0.1\),
\(\tau=2.5\times10^{-4}\), and \(T=0.3\).}

{For \(\underline G=I\) the drop does not change visibly up to \(T\).
For \(\underline G={\rm diag}(49,0.01,1)\) it changes shape although its film
energy is isotropic: the half-width of the wetted patch shrinks from \(0.77\)
to \(0.44\) in \(x\) but only to \(0.69\) in \(y\), the drop becomes taller,
and at \(T\) the contact line is close to an ellipse with axis ratio \(1.56\)
(Figure~\ref{fig:aniso-sub}).  The energy \(W_h^{\gamma,m}\) decreases at every
step, all substrate vertices lie within \(6\times10^{-11}\) of the sphere, and
\(\max R/r\le4.1\).}

\begin{figure}[htbp]
  \centering
  \begin{minipage}[b]{0.45\textwidth}\centering
    \includegraphics[width=0.5\linewidth]{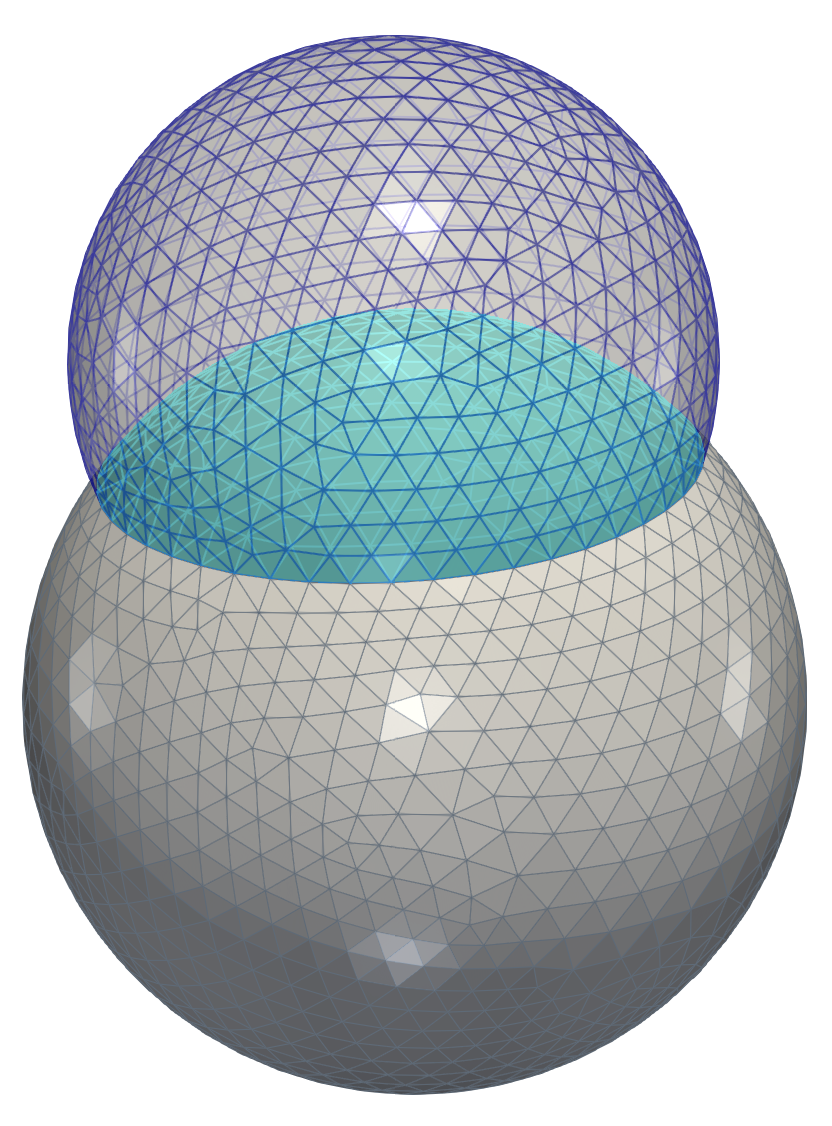}
    \par\smallskip{\small (a) \(\underline G=I\)\par}
  \end{minipage}\qquad
  \begin{minipage}[b]{0.45\textwidth}\centering
    \includegraphics[width=0.5\linewidth]{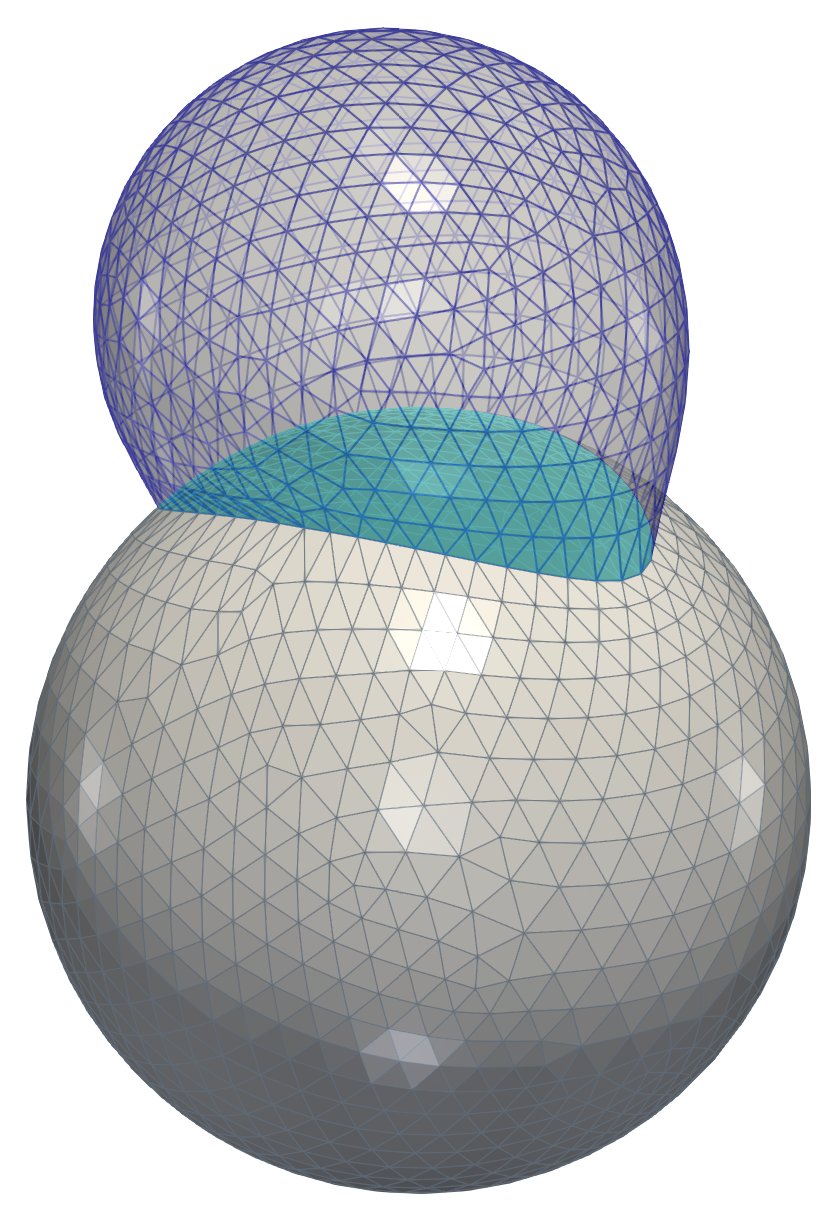}
    \par\smallskip{\small (b) \(\underline G={\rm diag}(49,0.01,1)\)\par}
  \end{minipage}\\[6pt]
  \begin{minipage}[b]{0.45\textwidth}\centering
    \includegraphics[width=0.62\linewidth]{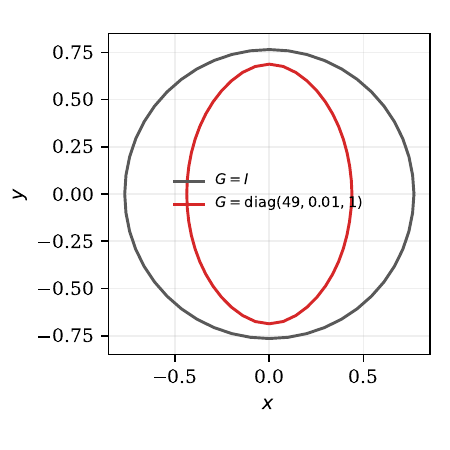}
    \par\smallskip{\small (c) contact lines, top view\par}
  \end{minipage}\qquad
  \begin{minipage}[b]{0.45\textwidth}\centering
    \includegraphics[width=0.85\linewidth]{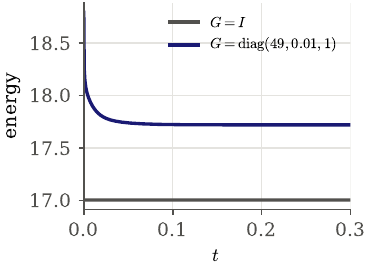}
    \par\smallskip{\small (d) total energy \(W_h^{\gamma,m}\)\par}
  \end{minipage}
  \caption{Isotropic drop on the sphere with
  \(\gamma_{\rm in}(\vec p)=\frac12\sqrt{\vec p\cdot\underline G\,\vec p}\) at
  \(T=0.3\).}
  \label{fig:aniso-sub}
\end{figure}
}
\section{Conclusions}
\label{sec:conclusion}

We have proposed parametric finite element methods for surface diffusion with
a contact line that moves on a curved substrate.
The linear schemes are
unconditionally energy stable and uniquely solvable.
Choosing the discrete substrate normals in a suitable way, that depends on the unknown solution,
leads to mildly nonlinear schemes that are still unconditionally stable and that, in addition,
can guarantee that the substrate vertices
stay exactly on the substrate.  The same holds for anisotropic film and anisotropic
substrate energies and for various discrete tangential motions: e.g.\ BGN, level-set BGN and dual-MDR.  The
numerical experiments confirm the energy decay and the nodal attachment, and compare the
level-set BGN and dual-MDR tangential motions.

\section*{Declarations}
ChatGPT-6 Astra was used to improve the language of the manuscript and to assist
with checking and debugging code.
The authors assume responsibility for all content.

\bibliographystyle{siamplain-nolinks}
\bibliography{manuscript_arxiv_version}

\end{document}